\documentclass[oneside]{amsart}

\usepackage[latin9]{inputenc}
\usepackage{xcolor}
\usepackage{amstext}
\usepackage{amsthm}
\usepackage{amssymb}
\PassOptionsToPackage{normalem}{ulem}
\usepackage{ulem}
\usepackage[unicode=true,pdfusetitle,
 bookmarks=true,bookmarksnumbered=false,bookmarksopen=false,
 breaklinks=false,pdfborder={0 0 1},backref=false,colorlinks=false]
 {hyperref}

\makeatletter

\providecommand{\tabularnewline}{\\}

\numberwithin{equation}{section}
\numberwithin{figure}{section}

\@ifundefined{date}{}{\date{}}
\usepackage{amsbsy}
\usepackage{amstext}
\usepackage{tikz}
\usepackage{multirow}

\usepackage{enumerate}
\usepackage{amsfonts}
\usepackage{graphicx}
\usetikzlibrary{arrows}
\usepackage[all]{xy}

\newtheorem{theorem}{Theorem}[section]
\newtheorem{proposition}[theorem]{Proposition}\newtheorem{lemma}[theorem]{Lemma}\newtheorem{corollary}[theorem]{Corollary}\newtheorem{example}{Example}

\newtheorem{definition}[theorem]{Definition}

\newtheorem{remark}[theorem]{Remark}

\newcommand{\To}{\Rightarrow} 
\newcommand{\ot}{\leftarrow}
\newcommand{\oT}{\Leftarrow}

\newcommand{\D}{\mathrm{D}}
\newcommand{\X}{\mathrm{X}}
\begin{document}
\title[BDL with strict implication and weak difference]{Bounded distributive lattices with strict implication and weak difference}
\author{Sergio Celani \and Agust\'{i}n Nagy \and William Zuluaga Botero}
\date{}

\begin{abstract}
In this paper we introduce the class of weak Heyting Brouwer algebras (WHB-algebras, for short).  We extend the well known duality between distributive lattices and Priestley spaces, in order to exhibit a relational Priestley-like duality for WHB-algebras. Finally, as an application of the duality, we build the tense extension of a WHB-algebra and we employ it as a tool for proving structural properties of the variety such as the finite model property, the amalgamation property, the congruence extension property and the Maehara interpolation property.
\end{abstract}

\maketitle

\section{Introduction}

Bounded lattices with additional operators occur often as algebraic models of Non-Classical Logics. This is the case of Boolean algebras which are the algebraic semantics of classical logic, Heyting algebras which model intuitionistic logic and Heyting Brouwer algebras which model bi-intuitionistic logic \cite{Wolter}. In all these cases, the operations $\vee$ and $\wedge$ are interpreted as the logical disjunction and conjunction, and the additional operations are usually interpretations of other logical connectives. Following this line and restricting to the case of logics in the language of intuitionistic logic, in \cite{CJ2} \emph{weak Heyting algebras} are introduced as a class of algebras that model the subintuitionistic logic $wK_{\sigma}$ \cite{CJ1}. They were defined as algebras \( (A,\wedge,\vee,\to,0,1) \) such that \( (A,\wedge,\vee,0,1) \) is a bounded distributive lattice and the map $ \to $, which is called the \emph{strict implication}, satisfies the following conditions for each \( a,b,c \in A \): 
\begin{itemize}
    \item[(1)] \( a \to a = 1,\)
    \item[(2)] \( a \to (b \wedge c) = (a \to b) \wedge (a \to  c), \)
    \item[(3)] \( (a \vee b) \to c = (a \to c) \wedge (b \to c), \)
    \item[(4)] \( (a \to b) \wedge (b \to c) \leq a \to c. \)
\end{itemize}
\noindent
Its follows from the definition that the class of all weak Heyting algebras forms a variety which is a generalization of several varieties of algebras like Heyting algebras and subresiduated lattices for instance.

In the same spirit, one may define the class of \emph{weak difference algebras} as algebras whose members are algebras \( (A,\wedge,\vee,\leftarrow,0,1) \) where $(A,\wedge,\vee,0,1)$ is a bounded distributive lattice and the map $\ot$, which is called \emph{weak difference}, satisfies the following conditions for every \( a,b,c \in A \):
\begin{itemize}
\item[(1)] \( a \leftarrow a = 0, \)
\item[(2)] \( (a \vee b) \leftarrow c = (a \leftarrow c) \vee (b \leftarrow c), \)
\item[(3)] \( a \leftarrow (b \wedge c) = (a \leftarrow b) \vee (a \leftarrow c), \)
\item[(4)] \( a \leftarrow c \leq (a \leftarrow b) \vee (b \leftarrow c). \)
\end{itemize} 
It follows from the definition that the class of all weak difference algebras forms a variety which provide us of a generalization of co-Heyting algebras \footnote{An algebra $(A,\wedge,\vee,\ot,0,1)$ is said to be a co-Heyting algebra provided that $(A,\wedge,\vee,0,1)$ is a bounded distributive lattice and for all $a,b,c \in A$ it holds the following residuation law: $a \ot b \leq c$ if and only if $a \leq b \vee c$}. \\

\vspace{1pt}

Heyting algebras and co-Heyting algebras are combined in \cite{Wolter} in order to study the algebraic semantics of intutionistic logics endowed with a binary conective $\ot$ which is called the co-implication. Thus, Heyting-Brouwer algebras are defined as algebras $(A,\wedge,\vee,\to,\ot,0,1)$ where $(A,\wedge,\vee,\to,0,1)$ is a Heyting algebra and $(A,\wedge,\vee,\ot,0,1)$ is a co-Heyting algebra. In order to give a suitable generalization of Heyting Brouwer algebras in this paper we shall study the variety of \emph{double weak Heyting algebra} (DWH-algebra, for short), whose members are algebras \( (\mathbf{A},\rightarrow,\leftarrow) \) such that:
\begin{itemize}
\item[(1)] \( (\mathbf{A},\rightarrow) \) is a weak Heyting algebra.
\item[(2)] \( (\mathbf{A},\leftarrow) \) is a weak difference algebra.
\end{itemize}
In particular, we shall study the subvariety of weak Heyting Brouwer algebras (WHB-algebras for short) whose members are double weak Heyting algebras that in addition the following conditions holds for each $a,b,c \in A$:
\begin{itemize}
\item[(3)] \( a \wedge ((a \to b) \rightarrow 0) \leq b, \)
\item[(4)] \( a \leq b \vee (1 \to (a \rightarrow b)). \)
\end{itemize}

Canonical examples of weak Heyting Brouwer algebras are Boolean algebras and Heyting Brouwer algebras. Taking into account the connection between Heyting Brouwer algebras and $S_{4}$ tense algebras given in \cite{Wolter} and inspired in the the study of WH-algebras we start the study of DWH-algebras and WHB-algebras providing a suitable generalization of dual Heyting algebras and Heyting Brouwer algebras respectively. The contribution of this paper is addressed to present a relational Priestley-like duality between the category of WHB-algebras and the category of WHB-spaces. Moreover we use the Priestley duality in order to  study the lattice of congruence of any WHB-algebras. Besides we shall study several properties of WHB-algebras by means of the tense companion of any WHB-algebra. This paper is organized as follows: in Section \ref{S2} we introduce the basic definitions and results that we will use along this paper. In Section \ref{S3} we shall study a relational representation for the variety of weak difference algebras and some subvarieties. In Section \ref{S4} we study the variety of WHB-algebras and some of its subvarieties. In Section \ref{S5} we develop a Priestley style duality for the category whose objects WHB-algebras and the arrows are algebra homomorphism. Finally in Section \ref{S6} we use the Priestley duality in order to obtain the tense companion of any WHB-algebra.

\section{Preliminaries} \label{S2}

In this section, we recall some basic facts and results that will be useful throughout this paper. We start this section by giving some elementary definitions. Then we move on to recall results that involve the variety of weak Heyting algebras and some of its subvarieties presented in \cite{CJ2}. \\

\vspace{1pt}

Let $X$ be a set we write $\mathcal{P}(X)$ to denote the power set of $X$. A partially ordered set, or poset, is a pair $(X,\leq)$ where $X$ is s non empty set and $\leq$ is a reflexive and transitive relation. Let $(X,\leq)$ be a poset. We say that \( U \) is an \emph{upset} if \( x \in U \) and \( x \leq y \) then \( y \in U \). Dually, we say that \( U \) is a \emph{downset} if \( x \in U \) and \( y \leq x \) then \( y \in U \). We write \( \mathrm{Up}(X) \) and \( \mathrm{Do}(X) \) to denote the set of all upsets and downsets of \( X \) respectively. If \( U \subseteq X \) we define the upset generated by \( U \) as follows:
\[ [U) = \{ x \in X \colon u \leq x \text{ for some } u \in U  \}. \]
Besides, we define the downset generated by \( U \) as follows:
\[ (U] = \{ x \in X \colon x \leq u \text{ for some } u \in U \}. \]
If $U = \{x\}$ then we write $[x)$ and $(x]$ in place of $[\{x\})$ and $(\{x\}]$ respectively. Let \( X \) be a set and \( R \subseteq X \times X \) a binary relation defined on \( X \). If \( U \subseteq X \) we write 
\[ R(U) = \{ R(x) \colon x \in U  \} \hspace{0.2cm} \text{ and } \hspace{0.2cm} R^{-1}(U)=\{ x \in X \colon R(x) \cap U \neq \emptyset \}. \]
If \( U = \{ x \} \) then we write \( R(x) \) and \( R^{-1}(x) \) instead of \( R(\{x\}) \) and \( R^{-1}(\{x\})\) respectively.\\

\vspace{0.2cm}

Let $\mathbf{L}$ be a bounded distributive lattice. A set $F \subseteq L$ is called {\it{filter}} if $1 \in F$, $F$ is increasing, and if $a,b \in F$, then $a \wedge b \in F$.  A proper filter $P$ is {\it{prime}} if for every $a,b \in A$, $a \vee b \in P$ implies $a \in P$ or $b \in P$. We write $X(\bf{L})$ the set of all prime filters of $\bf{L}$. Let $\sigma_{\bf{L}} \colon L \to \mathcal{P}(X(\bf{L}))$ be the map defined by $\sigma_{\bf{L}}(a)=\{ P \in X({\bf{L}}) \colon a \in P\}$. Such a map is called the \emph{Stone map} and the family $\sigma_{\bf{L}}[L]=\{ \sigma_{\bf{L}}(a) \colon a \in A\}$ is closed under unions, intersections, and contains $\emptyset$ and $L$, i.e., it is a bounded distributive lattice. Moreover, $\sigma_{\bf{L}}$ establishes an isomorphism between ${\bf{L}}$ and $\sigma_{\bf{L}}[L]$.

A {\it{Priestley space}} is a triple $( X, \leq, \tau )$ where $( X, \leq )$ is a poset and $( X, \tau )$ is a compact totally order-disconnected topological space. A morphism between Priestley spaces is a continuous and monotone function between them. If $( X, \leq, \tau )$ is a Priestley space, then the family of all clopen increasing sets is denoted by $D(X)$, and it is well known that $D(X)$ endowed with the union, the intersection, $X$ and $\emptyset$ is a bounded distributive lattice. The Priestley space of a bounded distributive lattice ${\bf{L}}$ is the triple $( X({\bf{L}}), \subseteq, \tau_{\bf{L}} )$, where $\tau_{\bf{L}}$ is the topology generated by taking as a subbase the family $\{ \sigma_{\bf{L}}(a) \colon a \in A \} \cup \{ \sigma_{\bf{L}}(a)^{c} \colon a \in A \}$, where $\sigma_{\bf{L}}(a)^{c} = X({\bf{L}}) - \sigma_{\bf{A}}(a)$. Therefore, ${\bf{L}}$ and $D(X({\bf{L}}))$ are isomorphic. If $( X, \leq, \tau )$ is a Priestley space, then the map $\epsilon_{X} \colon X \to X(D(X))$ defined by $\epsilon_{X}(x) = \{ U \in D(X) \colon x \in U \}$, for every $x \in X$, is a homeomorphism and an order-isomorphism. On the other hand, if $Y$ is a closed set of $X({\bf{L}})$, then the relation 
\begin{equation} \label{congruence-closed}
\theta(Y)=\{(a,b)\in L \times L \colon \sigma_{\textbf{L}}(a)\cap Y = \sigma_{\textbf{L}}(b) \cap Y \}
\end{equation}
is a congruence of ${\bf{L}}$ and the map $Y\mapsto \theta(Y)$ establishes an anti-isomorphism between the lattice of closed subsets of $X({\bf{L}})$ and the lattice of congruences of $\textbf{L}$. 

If $h \colon \mathbf{L} \to \mathbf{M}$ is a lattice homomorphism, then the map $X(h): X({\bf{M}}) \to X({\bf{L}})$ defined by $X(h)(P)=h^{-1}[P]$, for each $P \in X({\bf{M}})$, is a continuous and monotone function. Conversely, if $( X, \leq_{X}, \tau_{X} )$ and $( Y, \leq_{Y}, \tau_{Y} )$ are Priestley spaces and $f \colon X \to Y$ is a continuous and monotone function, then the map $D(f): D(Y) \to D(X)$ defined by $D(f)(U)=f^{-1}[U]$, for each $D(Y)$, is a homomorphism between bounded distributive lattices. Furthermore, there is a duality between the algebraic category of bounded distributive lattices with homomorphisms and the category of Priestley spaces with continuous and monotone functions (\cite{CLP,Priestley}).
\\

Let $\mathcal{K}$ be a class of algebras of a given type. We write $\mathbb{I}(\mathcal{K})$, $\mathbb{S}(\mathcal{K})$, $\mathbb{H}(\mathcal{K})$ and $\mathbb{P}(\mathcal{K})$  for the classes of isomorphic images, subalgebras, homomorphic images and products of elements of $\mathcal{K}$. We recall that due to the well known Birkhofff's Theorem, the variety generated by $\mathcal{K}$, $\mathbb{V}(\mathcal{K})$ has the form $\mathbb{H}\mathbb{S}\mathbb{P}(\mathcal{K})$. 

\subsection{Weak Heyting algebras and subvarieties} \label{S2.1}

In this section, we recall some basic results that involve weak Heyting algebras and some of its subvarieties like Heyting algebras, subresiduated lattices and basic algebras \cite{CJ2,Corsi,EH,Visser}. 

\begin{definition} \label{def WH algebra} An algebra \( \mathbf{A} = (A,\wedge,\vee,\to,0,1) \) of type $(2,2,2,0,0)$ is said to be a weak Heyting algebra, WH-algebra for short, if \( (A,\wedge,\vee,0,1) \) is a bounded distributive lattice and the following conditions are satisfied for each \( a,b,c \in A \):
\begin{enumerate}[\normalfont 1.]
    \item \( a \to a = 1\).
    \item \( a \to (b \wedge c) = (a \to b) \wedge (a \to  c) \).
    \item \( (a \vee b) \to c = (a \to c) \wedge (b \to c) \).
    \item \( (a \to b) \wedge (b \to c) \leq a \to c \).
\end{enumerate}
\end{definition}

It follows from Definition \ref{def WH algebra} that the class of all WH-algebras is a variety and we write \( \mathcal{WH} \) to denote it. Several subvarieties of \( \mathcal{WH} \) appear in the literature as algebraic counterpart of sublogics of the intutionistic logics. For instance the variety \( \mathcal{H} \) of Heyting algebras is a subvariety of \( \mathcal{WH} \). Moreover, let $\bf{A}$ be a WH-algebra and consider the following conditions for each $a,b,c \in A$:
\begin{itemize}
    \item[R] \( \colon a \wedge (a \to b) \leq b \)
    \item[T] \( \colon a \to b \leq c \to (a \to b)\)
    \item[B] \( \colon a \leq 1 \to a \)
\end{itemize}

We say that \( \mathbf{A} \) is a RWH-algebra if in addition the condition (R) is satisfied for each $a,b \in A$. We say that a WH-algebra $\bf{A}$ is a TWH-algebra if the condition (T) is satisfied for each $a,b,c \in A$. Subresiduated lattice were introduced in \cite{EH} as a pair \( (A,D) \) where \( A \) is a bounded distributive lattice and \( D \) is a bounded sublattice such that for each \( a,b \in A \) there is \( z \in D \) with the following property: for each \( d \in D \) \( a \wedge d \leq b \) if and only if \( d \leq z \). The element \( z \) is denoted by \( z = a \to b \). Thus, subresiduated lattices can be regarded as an algebra \( (A,\wedge,\vee,\to,0,1) \) where \( D = \{ a \in A \colon 1 \to a = a \} \). It was proved that \( \bf{A} \) is a subresiduated lattice if and only if \( \bf{A} \) is a RWH-algebra that in addition satisfies the condition (T) for each $a,b,c \in A$. WH-algebras that in addition satisfies the condition (B) for each $a \in A$ are called basic algebras. A basic algebra $\bf{A}$ that in addition satisfies the condition (R) for each $a,b \in A$ is a Heyting algebra.
We write \( \mathcal{RWH} \), \( \mathcal{TWH}\), \( \mathcal{SRL} \), \(\mathcal{B} \) and \( \mathcal{H} \) to denote the varieties of RWH-algebras, TWH-algebras, subresiduated lattices, basic algebras and Heyting algebras respectively. The order between these varieties of WH-algebras is represeted by the lattice of subvarieties given in Figure \ref{Fig: Lattice of Subvarieties}.

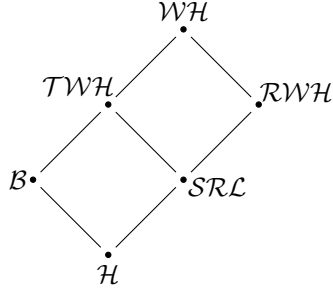
\begin{figure}[h]
\begin{tikzpicture}
\draw (-0.1,0.1)--(-0.90,0.9);
\draw (0.1,0.1)--(0.90,0.9);
\draw (-0.90,1.1)--(-0.1,1.90);
\draw (0.90, 1.1)--(0.1,1.90);
\draw (1.1, 1.1)--(1.9, 1.9);
\draw (1.90, 2.1)--(1.1,2.90);
\draw (0.1,2.1)--(0.9, 2.9);
\filldraw (0,0) circle (1pt);
\filldraw (1,1) circle (1pt);
\filldraw (-1,1) circle (1pt);
\filldraw (0,2) circle (1pt);
\filldraw (1,3) circle (1pt);
\filldraw (2,2) circle (1pt);
\draw (0,-0.25) node {$\mathcal{H}$};
\draw (1.45,0.9) node {$\mathcal{SRL}$};
\draw (2.5,2.17) node {$\mathcal{RWH}$};
\draw (-0.40,2.25) node {$\mathcal{TWH}$};
\draw (1,3.25) node {$\mathcal{WH}$};
\draw (-1.2,1.0) node {$\mathcal{B}$};
\end{tikzpicture}
\caption{Lattice of subvarieties} \label{Fig: Lattice of Subvarieties}
\end{figure}

Let \( \mathbf{A} \) be a WH-algebra. We write \( \mathrm{Fi}(\bf{A}) \) to denote the set of all filters of \( \bf{A} \) and \( X(\bf{A}) \) to denote the set of all prime filters of \( \bf{A} \). \\

\vspace{1pt}

Let \( \bf{A} \) be a WH-algebra. We shall consider the relation $R_{{\bf{A}}} \subseteq X({\bf{A}}) \times {\rm{Fi}}({\bf{A}})$  given in \cite[Section 3.1]{CJ2} which is defined as follows:
\begin{equation} \label{relacion WH}
    (P,Q) \in R_{{\bf{A}}} \text{ if and only if } (\forall a,b \in A)[a \to b \in P, a \in Q \Rightarrow b \in Q]
\end{equation}

\begin{remark} Some simple computations shows that $\bf{A}$ is a Heyting algebra if and only if $R_{{\bf{A}}}$ is equal to $\subseteq$.
\end{remark}

Let $\bf{A}$ be a weak Heyting algebra and $F\in \mathrm{Fi}(\bf{A})$. Let us consider the closure operator (see \cite[Proposition 3.4]{CJ2}) $D_{F}\colon\mathcal{P}(A)\rightarrow\mathcal{P}(A)$ defined as follows:
\[
D_{F}(X)=\left\{ b\in A:\exists Y\subseteq X\text{ finite such that } \bigwedge Y\rightarrow b \in F\right\} ,
\]
where $\bigwedge Y$ is the infimum of $Y$. If $Y = \emptyset$, we define $\bigwedge Y=1$.
Note that for every $X\subseteq A$, $D_{F}(X)$ is the least filter $\mathbf{A}$ such that $\left(F,H\right)\in R_{\mathbf{A}}$.
Besides for every $F,G \in \mathrm{Fi}(\mathbf{A})$  we have that
\[
(F,G)\in R_{\mathbf{A}} \text{ if and only if } D_{F}(G) = G.
\]

The following is a generalization of the Prime Filter Theorem \cite[Lemma
3.7]{CJ2}.

\begin{proposition} \label{TFP}
Let $\bf{A}$ be a weak Heyting algebra, let $F$ be a filter and $I$ an ideal of $\bf{A}$ and let $X\subseteq A$.
If $D_{F}(X)\cap I=\emptyset$, then there is a prime filter $P$
such that $D_{F}(X)\subseteq P$, $(F,P)\in R_{\bf{A}}$ and $P\cap I=\emptyset$.
\end{proposition}

As a direct application of Proposition \ref{TFP} we obtain the following results:

\begin{corollary} \label{TPF Corolario 2}
Let $\bf{A}$ be a weak Heyting
algebra. Let $P\in X(\bf{A})$ and $a,b \in A$. Then $a\rightarrow b\notin P$
if and only if there exists $Q\in X(\mathbf{A})$ such that $(P,Q)\in R_{\bf{A}}$, $a\in Q$
and $b\notin Q$.
\end{corollary}

Now we recall some results of representation of weak Heyting algebras and some of its subvarieties. 

\begin{definition}  A relational structure $\mathcal{F} = \left(X,\leq,R\right)$ is said to be a \emph{WH-frame} provided that the following conditions are satisfied:
\begin{enumerate}[\normalfont 1.]
\item $(X,\leq)$ is a partial ordered set. 
\item $R$ is a binary relation such that $\leq \, \circ \, R \, \subseteq R$.
\end{enumerate} 
\end{definition}

Let $\bf{A}$ be a WH-algebra. Some simple computations shows that 
\[ \mathcal{F}(\mathbf{A})=\left(X(\mathbf{A}),\subseteq,R_{\mathbf{A}}\right) \]
is a WH-frame, where $\subseteq$ is the inclusion and $R_{\mathbf{A}}$
is the relation defined in (\ref{relacion WH}). The WH-frame \( \mathcal{F}(\bf{A}) \) is called the \emph{frame associated to} \( \bf{A} \).\\

\vspace{1pt}

Let \( \mathcal{F} \) be a WH-frame. A direct computation shows that the algebra
\[ \mathcal{A}(\mathcal{F}) = (\mathrm{Up}(X),\cap,\cup,\To_{R},\emptyset,X), \]
is a WH-algebra which is called the \emph{algebra associated to} $\mathcal{F}$, where the map \( \To_{R} \) is defined as follows 
\[ U \To_{R} V = \{ x \in X \colon R(x) \cap U \subseteq V \}. \]

Taking into account Corollary \ref{TPF Corolario 2} it follows the following representation Theorem of \cite[Theorem 3.14]{CJ1}.

\begin{theorem}\label{Teo representacion WH}
Every WH-algebra \( \bf{A} \) is isomorphic to some subalgebra of $\mathcal{A}(\mathcal{F}(\bf{A}))$.    
\end{theorem}

Representation theorem for WH-algebras can be specialized for the subvarieties mentioned above. The following results can be found in \cite[Section 4]{CJ2}.

\begin{lemma} \label{aux1} Let $\bf{A}$ be a WH-algebra and $\mathcal{F}(\bf{A})$ its canonical frame. Then
\begin{enumerate}[\normalfont (1)]
    \item The condition R is valid on $\bf{A}$ if and only if $R_{\bf{A}}$ is reflexive.
    \item The condition T is valid on $\bf{A}$ if and only if $R_{\bf{A}}$ is transitive.
    \item The condition B is valid on $\bf{A}$ if and only if for each $P,Q \in X(\bf{A})$ we have that $(P,Q) \in R_{\bf{A}}$ implies that $P \subseteq Q$ 
\end{enumerate}
\end{lemma}

\begin{lemma} \label{aux2} Let $\mathcal{F}=(X,\leq,R)$ be a WH-frame and $\mathcal{A}(\mathcal{F})$ it algebra associated. Then
\begin{enumerate}[\normalfont (1)]
    \item The relation $R$ is reflexive if and only if the condition $\rm{R}$ is valid on $\mathcal{A}(\mathcal{F})$
    \item The relation $R$ is transitive if and only if the condition $\rm{T}$ is valid on $\mathcal{A}(\mathcal{F})$
    \item For each $x,y \in X$, $(x,y) \in S$ implies $x \leq y$ if and only if the condition $\rm{B}$ is valid on $\mathcal{A}(\mathcal{F})$ 
\end{enumerate}
\end{lemma}

Taking into account Theorem \ref{Teo representacion WH}, Lemma \ref{aux1}  and Lemma \ref{aux2} we have the following result: 

\begin{theorem} \label{Teo representacion subvariedades} If $\bf{A} \in \{\mathcal{RWH},\mathcal{TWH},\mathcal{SRL},\mathcal{B},\mathcal{H} \}$ then $\bf{A}$ is isomorphic to some subalgebra of $\mathcal{A}(\mathcal{F}({\bf{A}}))$.  
\end{theorem}

\section{Weak difference algebras} \label{S3}

In this section we shall introduce the variety of weak difference algebras. This variety of algebras is connected with the variety of co-Heyting algebras in the same sense that weak Heyting algebras are linked with Heyting algebras, namely the variety of weak difference algebras provide us of a generalization of co-Heyting algebras. Also, we shall prove a representation theorem for the variety of weak difference algebras by certain relational structures.

\begin{definition} \label{def WD-algebra} An algebra \( {\bf{A}} = \left(A,\wedge,\vee,\ot,0,1 \right) \) of type \( (2,2,2,0,0) \) is said to be a \emph{weak diference algebra}, WD-algebra for short, provided that \( (A,\wedge,\vee,0,1) \) is a bounded distributive lattice and the following conditions are satisfied for each \( a,b,c \in A \):
\begin{enumerate}[\normalfont 1.]
\item \( a \ot a = 0 \)
\item \( (a \vee b) \ot c = (a \ot c) \vee (b \ot c) \)
\item \( a \ot (b \wedge c) = (a \ot b) \vee (a \ot c) \)
\item \( a \ot c \leq (a \ot b) \vee (b \ot c) \)
\end{enumerate} 
\end{definition}

It follows from Definition \ref{def WD-algebra} that that the class of all WD-algebras forms a variety and we write $\mathcal{WD}$ to denote it. In the same sense that the map $\to$ is called strict implication for any WH-algebra $\bf{A}$ the map $\ot$ is called the weak difference for every WD-algebra $\bf{A}$.

\begin{lemma} \label{monotonia de la diferencia} Let \( A \) be a WD-algebra and \( a,b,c \in A \). If \( a \leq b \) then
\begin{enumerate}[\normalfont (1)]
\item \( a \ot c \leq b \ot c \).
\item \( c \ot b \leq c \ot a \).
\end{enumerate}
\end{lemma}
\begin{proof}
\( (1) \) Suppose that \( a \leq b \), so we have that \( b = b \vee a \) which implies that
\[ (a \vee b) \ot c = (a \ot c) \vee (b \ot c). \] 
Thus \( a \ot c \leq b \ot c \). A similar argument can be used in order to prove \( (2) \).  
\end{proof}

\begin{example} Let ${\bf{A}} = \left( \{0,a,1\},\wedge, \vee,\ot,0,1 \} \right)$ be the three element chain with $0 < a < 1$ and the binary map $\ot$ defined by $x \ot y = 0$ for all $x,y \in A$. Some simple computations shows that $\bf{A}$ is a WD-algebra. 
\end{example}

\begin{example} Recall that an algebra \( {\bf{A}} = (A,\wedge,\vee,\ot,0,1) \) is said to be a co-Heyting algebra (or dual Heyting algebra) provided that \( (A,\wedge,\vee,0,1) \) is a bounded distributive lattice and for every \( a,b,c \in A \) the dual residuation law it holds:
\[ a \ot b \leq c \text{ if and only if } a \leq b \vee c. \]
Co-Heyting algebras are examples of WD-algebras.
\end{example}

It is clear that the conditions of Definition \ref{def WD-algebra} are the dual of the conditions of Definition \ref{def WH algebra}. In particular, we shall consider the dual equations of the equations R,T and B respectively: 
\begin{itemize}
    \item[R$^{*} \colon $] \( a \leq b \vee (a \ot b)\)
    \item[T$^{*} \colon $] \(  (a \ot b) \ot c \leq (a \ot b) \)
    \item[B$^{*} \colon $] \( a \ot 0 \leq a \)
\end{itemize}
Taking into account the equations R\(^{*}\), T\(^{*}\) and B\(^{*}\) it can be defined the dual subvarieties of \( \mathcal{THW} \), \( \mathcal{RWH} \), \( \mathcal{SRL} \), \( \mathcal{B} \) and \( \mathcal{H} \) as follows:

\begin{definition} \label{def subvariedades WD} A WD-algebra $\bf{A}$ is said to be a \emph{RWD-algebra} provided that the condition $\rm{R}^{*}$ is valid on $A$. We say that a WD-algebra $\bf{A}$ is a \emph{TWD-algebra} provided that the condition $\rm{T}^{*}$ is valid on $A$. A \emph{basic WD-algebra} is a WD-algebra that in addition satisfies the condition $\rm{B}^{*}$. We say that a WD-algebra ${\bf{A}}$ is a \emph{co-subresiduated lattice} or \emph{dual subresiduated lattice} whenever the conditions $\rm{R}^{*}$ and $\rm{T}^{*}$ are valid on $\bf{A}$. A basic WD-algebra $\bf{A}$ is said to be a co-Heyting algebra if in addition the condition $\rm{R}^{*}$ is valid on $\bf{A}$.
\end{definition}

We write $\mathcal{TWD}$, $\mathcal{RWD}$, $\mathcal{SRL^{*}}$, $\mathcal{B^{*}}$ and $\mathcal{H^{*}}$ to denote the variety of TWD-algebras, RWD-algebras, dual subresiduated lattices, basic WD-algebras and co-Heyting algebras respectively. \\

\vspace{0.2cm}

Recall that a Kripke frame is a structure \( (X,\leq) \) where \( \leq \) is a reflexive and transitive relation. In \cite{Wolter} it is shown that the algebra
  \[ \mathcal{A}(\mathcal{F}) = (\mathrm{Up}(X),\cap,\cup,\Leftarrow,\emptyset,X), \]
is a co-Heyting algebra, where the map \( \Leftarrow \) is defined as follows:
\[ U \oT V = \{ x \in X \colon (x] \cap (U \setminus V) \neq \emptyset \} \] 
Following this line, in what follows we shall introduce certain relational structures, called WD-frames, which allow us obtain examples of WD-algebras. \\

\vspace{0.2cm}

\begin{definition} A relational structure \(  \mathcal{F} = (X,\leq,S) \) is said to be a \emph{WD-frame} provided that:
\begin{enumerate}[\normalfont 1.]
\item \( (X,\leq) \) is a Kripke frame. 
\item \( \leq^{-1} \circ \, S \, \subseteq S \).
\end{enumerate}  
\end{definition}

Let \( \mathcal{F} \) be a WD-frame. We shall consider the algebra
 \[ \mathcal{A}(\mathcal{F}) = (\mathrm{Up}(X),\cap,\cup,\Leftarrow_{S},\emptyset,X), \] 
 where the binary map $ \oT_{S} $ is defined as follows
\[ U \oT_{S} V = \{ x \in X \colon S(x) \cap (U \setminus V) \neq \emptyset \}. \] 
Some simple computations shows that $\mathcal{A}(\mathcal{F})$ is a WD-algebra and we call it the \emph{WD-algebra of the WD-frame} $\mathcal{F}$. \\

Let \( \bf{A} \) be a WD-algebra and \( (P,Q) \in   X(\bf{A}) \times {\rm{Fi}}(\bf{A}) \). We define the binary relation \( S_{\bf{A}} \) as follows: 
\begin{equation} \label{def relacion S}
(P,Q) \in S_{\bf{A}} \text{ if and only if } (\forall a,b \in A)[a \in Q,b \notin Q \Rightarrow a \ot b \in P] 
\end{equation}

\begin{remark} Note that if $\bf{A}$ is a co-Heyting algebra then the relation \( S_{\bf{A}} \) collapse with the relation \( \subseteq^{-1} \). In fact, suppose that \( \bf{A} \) is a co-Heyting algebra then the following conditions are satisfied for all \( a,b \in A \):
\begin{enumerate}
\item[$\rm{R^{*}}$] \( a \leq b \vee (a \ot b) \) 
\item[$\rm{B^{*}}$] \( a \ot 0 \leq a \)
\end{enumerate} 
Now let \( P,Q \in X(\bf{A}) \) with \( Q \subseteq P \) we shall prove that \( (P,Q) \in S_{\bf{A}} \). Suppose that \( a \in Q \) and \( b \notin Q \). By $\rm{R^{*}}$ we have that \( b \vee (a \ot b) \in Q \) which implies that \( b \in Q \) or \( a \ot b \in Q \). Since \( b \notin Q \) it follows that \( a \ot b \in Q \subseteq P \), so \( a \ot b \in P \).

Conversely suppose that \( (P,Q) \in S_{\bf{A}} \) we shall prove that \( Q \subseteq P \). Suppose that \( a \in Q \) and let \( 0 \notin Q \). By assumption we have that \(  a \ot 0 \in P \) which implies that \( a \in P \).  
\end{remark}

Let \( \bf{A} \) be a WD-algebra and \( P \in X(\bf{A}) \). We define the closure operator \( F_{P} \colon \mathcal{P}(A) \to  \mathcal{P}(A) \) as follows
\[ F_{P}(X) = \left\{ a \in A \colon \exists Y \subseteq_{f} X \colon \bigvee Y \ot a \notin P \right\} \]
where \( \bigvee Y = 0 \) if \( Y = \emptyset \). Some simple computations shows that the operator \( F_{p} \) is a closure operator. Moreover the relation \( S_{\bf{A}} \) can be characterizated by the operator \( F_{P} \). 

\begin{proposition} \label{prop aux 1}
 Let \( \bf{A} \) be a WD-algebra, \( P,Q \in X(\bf{A}) \) and \( X \subseteq A \). Then 
\begin{enumerate}[\normalfont (1)]
\item \( F_{P}(X) \in Fi(\bf{A}) \)
\item \( \left(P,F_{P}(X) \right) \in S_{\bf{A}} \)
\item \( (P,Q) \in S_{\bf{A}} \) if and only if \( F_{P}(Q) = Q \).
\end{enumerate}
\end{proposition}
\begin{proof}
\( (1) \) Let \( a \in F_{P}(X) \) and suppose that \( a \leq b \) we shall prove that \( b \in F_{P}(X) \). Since \( a \in F_{P}(X) \) there is a finite subset \( Y \subseteq X \) for which \( \bigvee Y \ot a \notin P \). Taking into account that \( a \leq b \) and by Lemma \ref{monotonia de la diferencia} we have that \( \bigvee Y \ot b \notin P \), so \( b \in F_{P}(X) \) and \( F_{P}(X) \) is an upset. Moreover if \( a,b \in F_{P}(X) \) then there are \( Y,Z \subseteq X \) finite such that \( \bigvee Y \ot a \notin P \) and \( \bigvee Z \ot b \notin P \). Consider \( W = Y \cap Z \) then we have that \( \bigvee W \ot a \notin P \) and \( \bigvee W \ot b \notin P \), hence 
\[ \left( \bigvee W \ot a \right) \vee \left(\bigvee W \ot b\right) = \bigvee W \ot (a \wedge b) \notin P \]
Thus, \( a \wedge b \in F_{P}(X) \) and \( F_{P}(X) \) is a filter.\\
\vspace{1pt}

\( (2) \) Suppose that \( a \in F_{P}(X) \) and \( b \notin F_{P}(X) \), we shall prove that \( a \ot b \in P \). By assumption \( \bigvee Y \ot a \notin P \) for some finite \( Y \subseteq X \) and \( \bigvee Y \ot b \in P \). Note that 
\[ \bigvee Y \ot b \leq \left( \bigvee Y \ot a \right) \vee (a \ot b) \]
and \( \bigvee Y \ot b  \in P \) then \( \bigvee Y \ot a \in P \) or \( a \ot b \in P \). Taking into account that \( \bigvee Y \ot a  \notin P \) it follows that \( a \ot b \in P \).\\
\vspace{1pt}

\( 3 \Rightarrow) \) It will be enough to prove that \( F_{P}(Q) \subseteq Q \). Let \( a \in F_{P}(Q) \) then there is \( q \in Q \) such that \( q \ot a \notin P \). Since \( (P,Q) \in S_{A} \) and \( q \in Q \) we have that \( a \in Q \). 

\( \Leftarrow) \) Suppose that \( a \in Q \) and \( b \notin Q \) we shall prove that \( a \ot b \in P \). Since \( Q = F_{P}(Q) \) then there is \( z \in Q \) such that \( z \ot a \notin P \) and \( z \ot b \in P \). Then we have that 
\[ z \ot b \leq (z \ot a) \vee (a \ot b), \]
so \( z \ot a \in P \) or \( a \ot b \in P \). Taking into account that \( z \ot a \notin P \) we have that \( a \ot b \in P \) as we desired.  \end{proof}

The following result is the dual of the Proposition \ref{TFP}

\begin{theorem} \label{DTPF} Let \( \bf{A} \) be a WD-algebra, \( P \in X(\bf{A}) \), \( I \in Id(\bf{A}) \) and \( X \subseteq A \) such that \( F_{P}(X) \cap I = \emptyset \). Then there is \( Q \in X(\bf{A}) \) such that \( (P,Q) \in S_{\bf{A}} \), \( X \subseteq Q \) and \( Q \cap I = \emptyset \).
\end{theorem}
\begin{proof}
Let us consider the family 
\[ \mathcal{F} = \{ G \in \mathrm{Fi}({\bf{A}}) \colon (P,G) \in S_{{\bf{A}}}, X \subseteq G \text{ and } G \cap I = \emptyset \}. \]
Notice that \( F_{P}(X) \in \mathcal{F} \), so \( \mathcal{F} \neq \emptyset \). Moreover, \( \mathcal{F} \) is inductive, so by Zorn Lemma there is a maximal element \( Q  \in \mathcal{F} \). Let us see that \( Q \) is prime. In fact, suppose that \( a \vee b \in Q \) and \( a,b \notin Q \). Consider the filters 
\[ Q_{a}= F_{P}(Q,a) = \{ x \in A \colon (q \vee a) \ot x \notin P \text{ for some } q \in Q \} \]
and
\[ Q_{b}= F_{P}(Q,b) = \{ x \in A \colon (q \vee b) \ot x \notin P \text{ for some } q \in Q \} \]
By the maximality of \( Q \) and by (2) of Proposition \ref{prop aux 1} we have that \( Q_{a} \cap I \neq \emptyset \) and \( Q_{b} \cap I \neq \emptyset \), so there is \( z_{1},z_{2} \in I \) and \( q_{1},q_{2} \in Q \) such that \( (a \vee q_{1}) \ot z_{1} \notin P \) and \( (b \vee q_{2}) \ot z_{2} \notin P \). Taking into account Lemma \ref{monotonia de la diferencia} it follows that 
\[ (q_{1} \vee a) \vee (q_{2} \vee b) \ot (z_{1} \vee z_{2}) \notin P, \]
which implies that \( z_{1} \vee z_{2} \in F_{P}(Q) \cap I \). Since \( (P,Q) \in S_{{\bf{A}}} \) it follows that \( F_{P}(Q) = Q \), so we have that \( Q \cap I \neq \emptyset \) which is a contradiction. Thus \( Q \) is prime, \( (P,Q) \in S_{{\bf{A}}} \), \( X \subseteq Q \) and \( Q \cap I = \emptyset \).
\end{proof}

\begin{corollary} \label{corolario DTFP} Let \( \bf{A} \) be a WD-algebra, \( P \in X(\bf{A}) \) and \( a,b \in \bf{A} \). Then \( a \ot b \in P \) if and only if there is \( Q \in X(\bf{A}) \) such that \( (P,Q) \in S_{\bf{A}} \), \( a \in Q \) and \( b \notin Q \).
\end{corollary}
\begin{proof}
Suppose that \( a \ot b \in P \) we shall prove that 
\[ F_{P}([a)) \cap (b] \neq \emptyset \]
In fact, suppose that there is \( x \in F_{P}([a)) \cap (b] \), so there is \( z \in [a) \) such that \( z \ot x \notin P \) and \( x \leq b \). Since \( a \leq z \) it follows that \( a \ot x \notin P \). Besides, since \( x \leq b \) we have that \( a \ot b \notin P \) which is absurd. Taking into account Theorem \ref{DTPF} there is \( Q \in X(\bf{A}) \) such that \( (P,Q) \in S_{\bf{A}} \), \( a \in Q \) and \( b \notin Q \).

The converse is immediate.
\end{proof}

\vspace{1pt}

\begin{lemma} \label{marco asociado al algebra} Let \( \bf{A} \) be a WD-algebra. Then the structure \( \mathcal{F}({\bf{A}})= \left(X({\bf{A}}),\subseteq,S_{{\bf{A}}} \right) \) is a WD-frame.
\end{lemma} 
\begin{proof}
It is clear that \( (X(\bf{A}),\subseteq) \) is a Kripke frame. Let \( (P,Q) \in \,\ \subseteq^{-1} \circ \, S_{\bf{A}} \), so there is \( Z \in X(\bf{A}) \) such that \( Z \subseteq P \) and \( (Z,Q) \in S_{\bf{A}} \). Let us see that \( (P,Q) \in S_{\bf{A}} \). Let \( a \in Q \) and \( b \notin Q \) then by assumption we have that \( a \ot b \in Z \subseteq P \). Thus \( a \ot b \in P \) and \( (P,Q) \in S_{\bf{A}} \).
\end{proof}

Let $\bf{A}$ be a WD-algebra. The structure $\mathcal{F}(\bf{A})=(X(\bf{A}),\subseteq,S_{\bf{A}})$ is called \emph{the WD-frame} of $\bf{A}$, where the relation $S_{\bf{A}}$ is the relation defined in \ref{def relacion S} 

\begin{theorem} \label{representacion WD} Every WD-algebra \( \bf{A} \) is isomorphic to some subalgebra of \( \mathcal{A}(\mathcal{F}(\bf{A})) \).
\end{theorem}
\begin{proof}
Let \( \sigma_{\bf{A}} \) be the Stone map. It is known that \( \sigma_{\bf{A}} \) is a homomorphism of bounded distributive lattices. Taking into account Corollary \ref{corolario DTFP} it follows that \( \sigma_{\bf{A}}(a \ot b) = \sigma_{\bf{A}}(a) \oT_{S_{\bf{A}}} \sigma_{\bf{A}}(b) \). Thus, \( \bf{A} \) is isomorphic with the subalgebra \( \sigma_{\bf{A}}[A] \). 
\end{proof}

In what follows we shall study the representation for the subvarieties of $\mathcal{WD}$ given in Definition \ref{def subvariedades WD}.

\begin{lemma} \label{ecuacion y condicion WD} Let $\bf{A}$ be a WD-algebra and $\mathcal{F}(\bf{A})$ its canonical WD-frame. Then
\begin{enumerate}[\normalfont (1)]
    \item The condition $\mathrm{R}^{*}$ is valid on $\bf{A}$ if and only if the relation $S_{\bf{A}}$ is reflexive.
    \item The equation $\mathrm{T}^{*}$ is valid on $\bf{A}$ if and only if the relation $S_{\bf{A}}$ is transitive
    \item The equation $\mathrm{B}^{*}$ is valid on $\bf{A}$ if and only if for each $P,Q \in X(\bf{A})$ we have that $ (P,Q) \in  S_{\bf{A}}$ imples $Q \subseteq P$.
\end{enumerate}
\end{lemma}
\begin{proof}
$1 \Rightarrow)$ Suppose that the equation $\mathrm{R}^{*}$ is valid on $\bf{A}$ and let $P \in X(\bf{A})$ and $a \in P$, $b \notin P$. By assumption, $a \leq b \vee (a \ot b) \in P$ so we have that $b \vee (a \ot b) \in P$. Since $b \notin P$ it follows that $a \ot b \in P$. Thus, $S_{\bf{A}}$ is reflexive.

$\Leftarrow)$ Suppose that $S_{\bf{A}}$ is reflexive we shall prove that the condition $\mathrm{R}^{*}$ is valid on $\bf{A}$. By contraposition, suppose that $\mathrm{R}^{*}$ is not valid on $\bf{A}$ then there are $a,b \in A$ such that $ a \nleq b \vee (a \ot b)$. Then there is $P \in X(\bf{A})$ such that $ a \in P$ and $b \vee (a \ot b) \notin P$, so $b \notin P$ and $a \ot b \notin P$ which implies that $(P,P) \notin S_{\bf{A}}$. Thus, $S_{\bf{A}}$ is not reflexive. \\

\vspace{1pt}

$2 \Rightarrow$ Suppose that the equation $\mathrm{T}^{*}$ is valid on $\bf{A}$ and let $P,Q,Z \in X(\bf{A})$ such that $(P,Q) \in S_{\bf{A}}$ and $(Q,Z) \in S_{\bf{A}}$. We shall prove that $(P,Z) \in S_{\bf{A}}$. Let $a \in P$ and $b \notin P$. Since $(P,Q) \in S_{\bf{A}}$ we have that $a \ot b \in Q$. Since $0 \notin Q$ and $(Q,Z) \in S_{\bf{A}}$ we have that $(a \ot b) \ot 0 \in Z$. Taking into account that the equation $\mathrm{T}^{*}$ is valid on $\bf{A}$ we have that $a \ot b \in Z$ which was our aim.

$\Leftarrow)$ By contraposition suppose that the equation  $\mathrm{T}^{*}$ is not valid on $\bf{A}$ then there are $a,b,c \in A$ such that $(a \ot b) \ot c \neq a \ot b$. Then there is $P \in X(\bf{A})$ such that $(a \ot b) \ot c \in P $ and $a \ot b \notin P$. By Corollary \ref{corolario DTFP} it follows that there is $Q \in X(\bf{A})$ such that $(P,Q) \in S_{\bf{A}}$, $a \ot b \in Q$ and $c \notin Q$. Again, by Corollary \ref{corolario DTFP} there is $Z \in X(\bf{A})$ such that $(Q,Z) \in S_{\bf{A}}$, $a \in Z$ and $b \notin Z$. Thus, $(P,Q) \in S_{\bf{A}}$, $(Q,Z) \in S_{\bf{A}}$ and $(P,Z) \notin S_{\bf{A}}$ which implies that $S_{\bf{A}}$ is not transitive.\\

\vspace{1pt}

$3 \Rightarrow)$ Suppose that the condition $a \ot 0 \leq a$ is valid for each $a \in A$ and let $P,Q \in S_{\bf{A}}$. Suppose that $a \in Q$. Since $0 \notin Q$ we have that $a \ot 0 \in P$. Taking into account that $a \ot 0 \leq a$ it follows that $a \in P$ which was our aim.

$\Leftarrow)$ By contraposition we assume that the condition $\mathrm{B}^{*}$ is not valid on $\bf{A}$ that is there is $a \in A$ such that $a \ot 0 \nleq a$. By prime filter theorem there is $P \in X(\bf{A})$ such that $a \ot 0 \in P$ and $a \notin P$. By Corollary \ref{corolario DTFP} there is $Q \in X(\bf{A})$ such that $(P,Q) \in S_{\bf{A}}$, $a \in Q$ and $0 \notin Q$. Thus, $(P,Q) \in S_{\bf{A}}$, $a \in Q$ and $a \notin P$. 
\end{proof}

\begin{lemma} \label{lema condicion ecuacion WD} Let $\mathcal{F}=(X,\leq,S)$ be a WD-frame. Then
\begin{enumerate}[\normalfont (1)]
\item The relation $S$ is reflexive if and only if the equation $\mathrm{R}^{*}$ is valid on $\mathcal{A}(\mathcal{F})$
\item The relation $S$ is transitive if and only if the equation $\mathrm{T}^{*}$ is valid on $\mathcal{A}(\mathcal{F})$
\item For every $x,y \in X$, $(x,y) \in S$ implies that $y \leq x$ if and only if the equation $\mathrm{B}^{*}$ is valid on $\mathcal{A}(\mathcal{F})$.
\end{enumerate}
\end{lemma}
\begin{proof}
 $1 \Rightarrow)$ Suppose that $S$ is reflexive and let $U,V \in \mathrm{Up}(X)$ we shall prove that $U \subseteq V \cup (U \oT_{S} V)$. Let $x \in U$ and suppose that $x \notin V$. Since $S$ is reflexive we have that $S(x) \cap \left( U \setminus V \right) \neq \emptyset$ which implies that $x \in \left( U \oT_{S} V \right)$ which was our aim.

 $\Leftarrow)$ Suppose that $S$ is not reflexive we shall prove that the equation $\mathrm{R}^{*}$ is not valid on $\mathcal{A}(\mathcal{F})$. Since $S$ is not reflexive then there is $x \in X$ such that $(x,x) \notin S$. Consider the up-sets $U=[x)$ and $V=(x]^{c}$. We shall prove that $U \not \subseteq V \cup \left(U \oT_{S} V \right)$. Indeed, it is clear that $x \in U$. Also, since $x \leq x$ we have that $x \notin V$. Now we shall prove that $x \notin  \left( U \oT_{S} V \right)$. Suppose that $ a \in S(x) \cap (U \setminus V)$ then $x \leq a$ and $a \leq x$, so $a = x$ and $(x,a) \in S$ which is a contradiction. Thus, $x \in U$ and $x \notin V \cup (U \oT_{S} V)$ which implies that $\mathrm{R}^{*}$ is not valid on $\mathcal{A}(\mathcal{F})$.

$2 \Rightarrow)$ Suppose that the relation $S$ is transitive and let $U,V$ and $W \in \mathrm{Up}(X)$, we shall prove that $ (U \oT_{S} V) \oT_{S} W \subseteq U \oT_{S} V $. In fact, let $x \in (U \oT_{S} V) \oT_{S} W$ then we have that there is $y \in X$ such that $(x,y) \in S$, $y \in U \oT_{S} V$ and $y \notin W$. Since $y \in U \oT_{S} V$ there is $z \in X$ such that $(y,z) \in S$. $z \in U$ and $z \notin V$. By assumption we have that $(x,z) \in S$, $z \in U$ and $z \notin V$ which implies that $x \in U \oT_{S} V$ which was our aim. 

$\Leftarrow)$ By contraposition suppose that the relation $S$ is not transitive. Then there are $x,y,z \in S$ such that $(x,y) \in S$, $(y,z) \in S$ and $(x,z) \notin S$. Let us consider the upsets $U=[z)$, $V=(z]^{c}$ and $W=(x]^{c}$. We shall prove that $x \in (U \oT_{S} V) \oT_{S} W$ and $x \notin U \oT_{S} V$. Indeed, since $y \in S(x)$, $z \in S(y)$ and $z \in U \setminus V$ it follows that $ y \in U \oT_{S} V$ and $S(x) \cap (U \oT_{S} V) \neq \emptyset$. Also, it is clear that $x \notin W$, so $ x \in (U \oT_{S} V) \oT_{S} W$. However, $x \notin U \oT_{S} V$. In fact, if $a \in S(x) \cap (U \setminus V)$ we have that  $(x,a) \in S$ and $z = a$ which is absurd. Thus, $ x \notin U \oT_{S} V$ and the equation $\mathrm{T}^{*}$ is not valid on $\mathcal{A}(\mathcal{F})$.\\

\vspace{1pt}

$3 \Rightarrow)$ Suppose that $(x,y) \in S$ implies $y \leq x$ for every $x,y \in X$ and let $U \in \mathrm{Up}(X)$. We shall prove that $U \oT_{S} \emptyset \subseteq U$. In fact, if $x \in U \oT_{S} \emptyset$ then $S(x) \cap U \neq \emptyset$, so there is $y \in X$ such that $(x,y) \in S$ and $y \in U$. By our assumption we have that $y \leq x$. Taking into account that $U$ is an up-set we have that $x \in U$ which was our aim.

$\Leftarrow) $ Suppose that $(x,y) \in S$ and $y \nleq x$. We shall prove that the condition $\mathrm{B}^{*}$ is not valid on $\mathcal{A}(\mathcal{F})$ that is $U \oT_{S} \emptyset \not \subseteq U$ for some $U \in \mathrm{Up}(X)$. Consider the up-set $U = (x]^{c}$. Note that $y \in S(x) \cap U$ so $x \in U \oT_{S} \emptyset$, but $x \notin U$ so the condition $\mathrm{B}^{*}$ is not valid on $\mathcal{A}(\mathcal{F})$. 
\end{proof}

Taking into account lemmas \ref{lema condicion ecuacion WD}, \ref{ecuacion y condicion WD} and Theorem \ref{representacion WD} we have the following result. 

\begin{corollary} \label{Teo representacion subvariedades WD} If $\bf{A} \in \{ \mathcal{RWD},\mathcal{TWD},\mathcal{SRL}^{*},\mathcal{B}^{*},\mathcal{H}^{*} \}$ then $\bf{A}$ is isomorphic to some subalgebra of $\mathcal{A}(\mathcal{F}(\bf{A}))$.
\end{corollary}

\section{The variety of weak Heyting Brouwer algebras} \label{S4}

Recall that an algebra ${\bf{A}} = \left(A,\wedge,\vee,\to,\ot,0,1 \right)$ is a Heyting Brouwer algebra provided that $(A,\wedge,\vee,\to,0,1)$ is a Heyting algebra and $(A,\wedge,\vee,\ot,0,1)$ is a co-Heyting algebra. In this section we shall study the variety of algebras obtained by consider bounded distributive lattices with strict implication and weak difference which will be called double weak Heyting algebras algebras (DWH-algebras for short). In particular, we will consider a certain subvariety of the variety of DWH-algebras which is a suitable generalization of Heyting brouwer algebras and it has a strong connection with modal tense algebras.   

\begin{definition} \label{Definicion SWD} An algebra \( {\bf{A}} = (A,\wedge,\vee,\to,\ot,0,1) \) is said to be a \emph{Double weak  Heyting algebra}, DWH-algebra for short, provided that:
\begin{enumerate}[\normalfont 1.]
\item \( (A,\wedge,\vee,\to,0,1) \) is a weak Heyting algebra.
\item \( (A,\wedge,\vee,\ot,0,1) \) is a weak difference algebra.
\end{enumerate}
\end{definition}  

Taking into account Definition \ref{def WH algebra} and \ref{def WD-algebra} it follows that the class of all DWH-algebras is a variety and we write \( \mathcal{DWH} \) to denote the variety of DHW-algebras. Note that the variety $\mathcal{HB}$ is a subvariety of $\mathcal{DWH}$. Moreover the following example shows that the variety $\mathcal{HB}$ is a proper subvariety of $\mathcal{DWH}$.

\begin{example} \label{ex1} Let ${\bf{A}} = \left( \{0,a,1\}, \wedge,\vee,\to,\ot,0,1 \right)$ be the three elements chain with $0 < a < 1$ and the maps $ \to $ and $\ot$ defined as follows: for each $x,y \in A$, $x \to y = 1$ and $x \ot y = 0$. Some simple computations shows that $\bf{A}$ is a DWH-algebra. Besides, since $ 1 \to a = 1 \nleq a$ it follows that ${\bf{A}} \notin \mathcal{HB}$
\end{example} 

\vspace{1pt}

In what follows we shall consider the following equations in the signature of DWH-algebras:
\begin{itemize}
\item[E1:] \( a \wedge ((a \to b) \ot 0) \leq b \)
\item[E2:] \( a \leq b \vee (1 \to (a \ot b)) \)
\end{itemize}

\begin{remark} \label{observacion sobre E1 y E2} Note that the equations \( \rm{E1} \) and \( \rm{E2} \) are valid on any HB-algebra $\bf{A}$. In fact suppose that \( \bf{A} \) is a HB-algebra. Taking into account the residuation law and the dual residuation law we have that 
\begin{eqnarray*}
 a \wedge ((a \to b) \ot 0) \leq b & \text{ if and only if } & (a \to b) \vee 0 \leq a \to b \\
                                   & \text{ if and only if } & a \to b \leq a \to b. 
\end{eqnarray*}
Also note that
\begin{eqnarray*}
a \leq b \vee (1 \to (a \to b)) & \text{ if and only if } & a \ot b \leq 1 \to (a \ot b) \\
                                & \text{ if and only if } & a \ot b \leq a \ot b 
\end{eqnarray*} 
\end{remark}

\begin{definition} A DWH-algebra $\bf{A}$ is said to be a \emph{weak Heyting Browver-algebra}, or WHB-algebra for short, if in addition the conditions $\rm{E1}$ and $\rm{E2}$ are satisfied.
\end{definition}

We write \( \mathcal{WHB} \) to denote the variety of WHB-algebras. It is immediate that the variety $\mathcal{WHB}$ is a subvariety of $\mathcal{DWH}$. Moreover the following example shows that $\mathcal{WHB}$ is a proper subvariety of $\mathcal{DWH}$. 

\begin{example} \label{ex2} Let $A$ be the bounded distributive lattice of four elements 

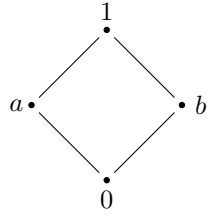
\begin{center}
\begin{figure}[h]
\begin{tikzpicture}
\draw (-0.1,0.1)--(-0.90,0.9);
\draw (0.1,0.1)--(0.90,0.9);
\draw (-0.90,1.1)--(-0.1,1.90);
\draw (0.90, 1.1)--(0.1,1.90);
\filldraw (0,0) circle (1pt);
\filldraw (1,1) circle (1pt);
\filldraw (-1,1) circle (1pt);
\filldraw (0,2) circle (1pt);
\draw (0,-0.25) node {$0$};
\draw (1.25,1.0) node {$b$};
\draw (0,2.25) node {$1$};
\draw (-1.2,1.0) node {$a$};
\end{tikzpicture}
\end{figure}
\end{center}

Also, consider the following maps:
\\

\vspace{1pt}

\begin{minipage}{0.48 \textwidth}
\begin{center}
\begin{tabular}{l|llll}
$\to$  & $0$  & $a$  & $b$  & $1$ \tabularnewline
\hline
$0$  & $1$  & $1$  & $1$  & $1$ \tabularnewline
$a$  & $1$  & $1$  & $1$  & $1$ \tabularnewline
$b$  & $b$  & $b$  & $1$  & $1$ \tabularnewline
$1$  & $b$  & $b$  & $1$  & $1$ \tabularnewline
\end{tabular}
\end{center}
\end{minipage}
\begin{minipage}{0.48 \textwidth}
\begin{center}
\begin{tabular}{l|llll}
$\ot$  & $0$  & $a$  & $b$  & $1$ \tabularnewline
\hline
$0$  & $0$  & $0$  & $0$  & $0$ \tabularnewline
$a$  & $a$  & $0$  & $a$  & $0$ \tabularnewline
$b$  & $0$  & $0$  & $0$  & $0$ \tabularnewline
$1$  & $a$  & $0$  & $a$  & $0$ \tabularnewline
\end{tabular}
\end{center}
\end{minipage}

\vspace{1pt}

Some simple computations shows that $\bf{A} = \left(A,\to,\ot\right)$ is a DWH-algebra. However $\rm{E1}$ nor $\rm{E2}$ are valid on $\mathbf{A}$. In fact, note that $a \wedge ((a \to b) \ot 0) = a$ and $a \nleq b$, so $\rm{E1}$ is not valid. Besides note that $b \vee (1 \to (a \ot b)) = b$ and $a \nleq b$, so $\rm{E2}$ is not valid on $\bf{A}$.
\end{example}

Taking into account examples \ref{ex1} and \ref{ex2} it follows that 
\[ {\mathcal{HB}} \subset {\mathcal{WHB}} \subset {\mathcal{DWH}}. \]
In what follows we shall study the representation of WHB-algebras by certain Kripke frames.

\begin{definition} A \emph{WHB-frame} is a structure \( (X,\leq,R,S) \) such that:
\begin{enumerate}[\normalfont 1.]
\item \( (X,\leq,R) \) is a WH-frame,
\item \( (X,\leq,S) \) is a WD-frame and 
\item \( S = R^{-1} \)
\end{enumerate}
\end{definition}

\begin{theorem} \label{PSW-algebra asociada al frame} If \( \mathcal{F} = (X,\leq,R,S) \) is a WHB-frame then the algebra 
\[ \mathcal{A}(\mathcal{F})= \left( \mathrm{Up}(X),\cap,\cup,\Rightarrow_{R},\oT_{S},\emptyset,X\right) \]
is a WHB-algebra.
\end{theorem}
\begin{proof}
Taking into account results of Section \ref{S2} and Section \ref{S3} it follows that \( \mathcal{A}(\mathcal{F}) \) is a DWH-algebra. We only need to prove that the equations \( \rm{E1} \) and \( \rm{E2} \) are satisfied on $\mathcal{A}(\mathcal{F})$. Let \( U,V \in \mathrm{Up}(X) \) we will prove that \( U \cap ((U \To_{R} V) \oT_{S} \emptyset) \subseteq V \) and \( U \subseteq V \cup (X \To_{R}(U \oT_{S} V)) \).

In fact, suppose that \( x \in U \cap ((U \To_{R} V) \oT_{S} \emptyset) \) then \( x \in U \) and we have that \( S(x) \cap (U \To_{R} V) \neq \emptyset \), so there is \( y \in S(x) \) and \( R(y) \cap U \subseteq V \). Since \( S = R^{-1} \) we have that \( (y,x) \in R \) and \( x \in U \) which implies that \( x \in V \). Thus the equation \( \rm{E1} \) is satisfied. 

Now, let \( x \in U \) and suppose that \( x \notin V \) we shall prove that \( x \in X \To_{R}(U \oT_{S} V) \) i.e, \( R(x) \subseteq U \oT_{S} V \). Let \( y \in R(x) \) then \( (x,y) \in R \) which implies that \( (y,x) \in S \) and \( x \in U \setminus V \) so we have that \( S(y) \cap (U \setminus V) \neq \emptyset \) and \( y \in U \oT_{S} V \) as we claim. Thus the equation \( \rm{E2} \) is satisfied too. 
\end{proof}

Let \( \bf{A} \) be a WHB-algebra. We shall consider the structure
\[ \mathcal{F}({\bf{A}})=(X({\bf{A}}),\subseteq,R_{\bf{A}},S_{\bf{A}}). \]

\begin{lemma} \label{ecuaciones y relaciones} Let \( \bf{A} \) be a WHB-algebra. Then 
\begin{enumerate}[\normalfont (1)]
\item \(  \rm{E1} \) is valid on \( \bf{A} \) if and only if \( S_{\bf{A}} \subseteq R_{\bf{A}}^{-1} \).
\item \( \rm{E2} \) is valid on \( \bf{A} \) if and only if \( R_{\bf{A}} \subseteq S_{\bf{A}}^{-1} \).
\end{enumerate}
\end{lemma}
\begin{proof}
(1) \( \Rightarrow \) Suppose that the equation \( \rm{E1} \) is valid on \( \bf{A} \) and let \( (P,Q) \in S_{\bf{A}} \) we shall prove that \( (Q,P) \in R_{\bf{A}} \). Let $a,b \in A$ such that \( a \to b \in Q \) and \( a \in P \). Since \( a \to b \in Q \), \( 0 \notin Q \) and \( (P,Q) \in S_{\bf{A}} \) we have that \( (a \to b) \ot 0 \in P \). Taking into account that \( a \in P \) we have that \( a \wedge ((a \to b) \ot 0) \in P \). Since \( \rm{E1} \) is valid on \( \bf{A} \) we have that \( b \in P \), which implies that \( (Q,P) \in R_{\bf{A}} \).

$\Leftarrow)$ By contraposition, suppose that the equation \( \rm{E1} \) is not valid on \( \bf{A} \). Then there is \( a,b \in A \) such that \( a \wedge ((a \to b) \ot 0) \nleq b \), so there is \( P \in X(\bf{A}) \) such that \( a \in P \), \( (a \to b) \ot 0 \in P \) and \( b \notin P \). Since \( (a \to b) \ot 0 \in P \), by Corollary \ref{corolario DTFP} there is \( Q \in X(\bf{A}) \) such that \( (P,Q) \in S_{\bf{A}} \) and \( a \to b \in Q \). Then, for every \( Z \in X(\bf{A}) \) such that \( (Q,Z) \in R_{\bf{A}} \), if \( a \in Z \) implies \( b \in Z \). By assumption, we have that \( (P,Q) \in S_{\bf{A}} \subseteq R_{\bf{A}}^{-1} \), so we have that \( (Q,P) \in S_{\bf{A}} \) and \( a \in P \) then we have that \( b \in P \) which is a contradiction. Thus, the equation \( \rm{E1} \) is valid on \( \bf{A} \).\\

\vspace{1pt}

(2) $\Rightarrow)$ Suppose that the equation \( \rm{E2} \) is valid on \( \bf{A} \) and let \( (P,Q) \in R_{\bf{A}} \). We shall prove that \( (Q,P) \in S_{\bf{A}} \). In fact, let $a,b \in A$ such that \( a \in P \) and \( b \notin P \) we need to prove that \( a \ot b \in Q \). Since \( a \in P \) we have that \( b \vee 1 \to (a \ot b) \in P \). Since \( P \) is prime and \( b \notin P \) we have that \( 1 \to (a \ot b) \in P \). Taking into account that \( (P,Q) \in R_{\bf{A}} \), \( 1 \to (a \ot b) \in P \) and \( 1 \in Q \) we have that \( a \ot b \in Q \) which was our aim. 

$\Leftarrow)$ By contraposition, suppose that the equation \( \rm{E2} \) is not valid on \( \bf{A} \), so there are \( a,b \in A \) such that \( a \nleq b \vee (1 \to (a \ot b)) \). Then there is \( P \in X(\bf{A}) \) such that \( a \in P \), \( b \notin P \) and \( 1 \to (a \ot b) \notin P \). Since \( 1 \to (a \ot b) \notin P \) there is \( Q \in X(\bf{A}) \) such that \( (P,Q) \in R_{\bf{A}} \) and \( a \ot b \notin Q \). By Corollary \ref{corolario DTFP} we have that for every \( Z \in X(\bf{A}) \), if \( (Q,Z) \in S_{\bf{A}} \) and \( a \in Z \) then \( b \in Z \). Taking into account that \( (P,Q) \in R_{{\bf{A}}} \subseteq S_{\bf{A}}^{-1} \) we have that \( (Q,P) \in S_{\bf{A}} \) and \( a \in P \), so \( b \in P \) which is a contradiction. Thus, the equation \( \rm{E2} \) is valid on \( \bf{A} \).     
\end{proof}

Taking into account lemmas \ref{aux1}, \ref{ecuacion y condicion WD} and \ref{ecuaciones y relaciones} we have the following result

\begin{corollary} \label{subvariedades de WHB} If $\bf{A}$ is a WHB-algebra then 
\begin{enumerate}[\normalfont 1.]
\item The condition $\rm{R}$ is valid on $\bf{A}$ if and only if the condition $\rm{R}^{*}$ is valid on $\bf{A}$
\item The condition $\rm{T}$ is valid on $\bf{A}$ if and only if the condition $\rm{T}^{*}$ is valid on $\bf{A}$
\item The condition $\rm{B}$ is valid on $\bf{A}$ if and only if the condition $\rm{B}^{*}$ is valid on $\bf{A}$
\end{enumerate}
\end{corollary}

Taking into account Lemma \ref{ecuaciones y relaciones} it follows that if \( \bf{A} \) is a WHB-algebra then the structure \( \mathcal{F}({\bf{A}})=(X({\bf{A}}),\subseteq,R_{{\bf{A}}},S_{{\bf{A}}}) \) is a WHB-frame. Moreover we have the following result:

\begin{theorem} \label{representacion WHB} Every WHB-algebra \( \bf{A} \) is isomorphic to some subalgebra of the algebra \( \mathcal{A}(\mathcal{F}(\bf{A})) \).
\end{theorem}

\section{Priestley Duality for WHB-algebras} \label{S5}

The aim of this section is to study a Priestley style duality for the category \( \mathsf{WHB} \) whose objects are WHB-algebras and arrows are algebra homomorphism. We start this section by introducing the category \( \mathsf{WHBS} \) whose objects are certain Priestley spaces endowed with two binary relations and the arrows are continuous order-preserving maps such that satisfy some additional conditions.   

\begin{definition}\label{def WHB-space} A structure \( (X,\tau,\leq,,R,S) \) is said to be a \emph{WHB-space} provided that:
\begin{enumerate}[\normalfont 1.]
\item \( (X,\tau,\leq) \) is a Priestley space.
\item \( (X,\leq,R,S) \) is a WHB-frame.
\item \( S(x) \) and \( R(x) \) are \( \tau \)-closed for every \( x \in X \).
\item For each \( U \) clopen the set \( R^{-1}(U) \) and \( S^{-1}(U) \) are clopen.
\end{enumerate} 
A map \( f \colon (X_{1},\tau_{1},\leq_{1},R_{1},S_{1}) \to (X_{2},\tau_{2},\leq_{2},R_{2},S_{2}) \) is said to be a \emph{WHBS-morphism} provided that
\begin{enumerate}[\normalfont 1.]
\item \( f \) is a continuous order-preserving map 
\item If \( (x,y) \in \mathcal{R}_{1} \) then \( (f(x),f(y)) \in \mathcal{R}_{2} \) for \( \mathcal{R} \in \{R,S \} \).
\item If \( (f(x),z) \in \mathcal{R}_{2} \) then there is \( y \in \mathcal{R}_{1}(x) \) such that \( f(y)=z \), for \( \mathcal{R} \in \{R,S\} \).  
\end{enumerate}
\end{definition}

We write \( \mathsf{WHBS} \) to denote the category whose objects are the WHB-spaces and the arrows are the WHB-morphism with the usual composition of functions.\\

\vspace{1pt}

Let \( (X,\leq,\tau,R,S) \) be a WHB-space. In particular, $(X,\tau,\leq)$ is a Priestley space and we write \( D(X) \) to denote the the family of all clopen upsets of \( (X,\tau,\leq) \). We shall consider the algebra 
\[ \D(X) = \left( D(X),\cap,\cup,\Rightarrow_{R},\oT_{S},\emptyset,X \right). \]

\begin{lemma} \label{Objetos PSWS-PSW} If \( (X,\leq,\tau,R,S) \) is a WHB-space then \( \D(X) \) is a WHB-algebra. 
\end{lemma}
\begin{proof}
Taking into account the Priestley duality for bounded distributive lattice it follows that \( (D(X),\cap,\cup,\emptyset, X) \) is a bounded distributive lattice. It will be enough prove that \( U \Rightarrow_{R} V \) and \( U \oT_{S} V \in \D(X) \) for all \( U, V \in \D(X) \). Note that \( U \Rightarrow_{R} V = X \setminus R^{-1}(U \setminus V) \) and \( U \oT_{S} V = S^{-1}(U \setminus V) \). Since \( U \setminus V \) is clopen it follows that \( U \Rightarrow_{R} V \) and \( U \oT_{S} V \) are clopen too. Besides since \( (X,\leq,R,S) \) is a WHB-frame then we have that \( U \Rightarrow_{R} V \) and \( U \oT_{S} V \) are upsets. Thus \( U \Rightarrow_{R} V \) and \( U \oT_{S} V \in \D(X) \). 

The fact that the identities E1 and E2 are valid in \( \D(X) \) follows from Theorem \ref{PSW-algebra asociada al frame} and from the fact that \( \D(X) \) is a subalgebra of \( \mathrm{Up}(X) \). Thus, the algebra \( \D(X) \) is a WHB algebra. 
\end{proof}

\begin{lemma} \label{flechas PSWS - PSW} If \( f \colon (X_{1},\tau_{1},\leq_{1},R_{1},S_{1}) \to (X_{2},\tau_{2},\leq_{2},R_{2},S_{2}) \) is an arrow in the category \( \mathsf{WHBS} \) then the map \( \D(f) \colon \D(X_{2}) \to \D(X_{1}) \) defined by \( \D(f)(U)=f^{-1}(U) \) is an arrow in $\mathsf{WHB}$. 
\end{lemma}
\begin{proof}
By results of Priestley duality for bounded distributive lattices we have that \( \D(f) \) is a homomorphism of bounded distributive lattices. Also it follows from \cite[Theorem 4.14]{CJ1} that \( \D(f)(U \Rightarrow_{R_{2}} V) = \D(f)(U) \Rightarrow_{R_{1}} \D(f)(V) \). Hence we only show \( \D(f)(U \oT_{S_{2}} V) = \D(f)(U) \oT_{S_{1}} \D(f)(V) \). 

Let \( x \in \D(f)(U \oT_{S_{2}} V) \) we shall prove that \( S_{1}(x) \cap f^{-1}(U \setminus V) \neq \emptyset \). In fact, if \( x \in \D(f)(U \oT_{S_{2}} V) \) then \( f(x) \in U \oT_{S_{2}} V \), so \( S_{2}(f(x)) \cap (U \setminus V) \neq \emptyset \). Then there is \( z \in S_{2}(f(x)) \) and \( z \in U \setminus V \). By property (2) of WHBS-morphisms there is \( y \in S_{1}(x) \) such that \( f(y) = z \). Thus \( y \in S_{1}(x) \cap f^{-1}(U \setminus V) \), so we conclude that \( x \in \D(f)(U) \oT_{S_{1}} \D(f)(V) \). Conversely, suppose that \( x \in \D(f)(U) \oT_{S_{1}} \D(f)(V) \) we shall prove that \( x \in \D(f)(U \oT_{S_{2}} V) \). In fact, if \( x \in \D(f)(U) \oT_{S_{1}} \D(f)(V) \) then there is \( y \in S_{1}(x) \cap f^{-1}(U \setminus V) \). By property (1) of WHBS-morphism we have that \( f(y) \in S_{2}(f(x)) \cap (U \setminus V) \), therefore we have that \( x \in \D(f)(U \oT_{S_{2}} V) \). Thus, we have that \(  \D(f)(U \oT_{S_{2}} V) = \D(f)(U) \oT_{S_{1}} \D(f)(V) \) and \( \D(f) \) is a homomorphism of WHB-algebras as we claim.
\end{proof}

Some simple computations show that we have defined a contravariant functor \( \D \colon \mathsf{WHBS} \to \mathsf{WHB} \) whose action on objects is described in Lemma \ref{Objetos PSWS-PSW} and its action on arrows is described in Lemma \ref{flechas PSWS - PSW}. In order to obtain a duality between \( \mathsf{WHB} \) and \( \mathsf{WHBS} \), in what follows we shall introduce a contravariant functor \( \mathrm{X} \colon \mathsf{WHB} \to \mathsf{WHBS} \) and natural isomorphism \( \sigma \colon \mathrm{I}_{\mathsf{WHB}} \to (\D \circ \mathrm{X}) \) and \( \epsilon \colon \mathrm{I}_{\mathsf{WHBS}} \to (\mathrm{X}\circ \D) \).\\

\vspace{1pt}

Let \( \bf{A} \) be a WHB-algebra. We consider the structure 
\[ \mathrm{X}({\bf{A}})=(X({\bf{A}}),\tau_{{\bf{A}}},\subseteq,R_{{\bf{A}}},S_{{\bf{A}}}), \]
where \( (X({\bf{A}}),\subseteq,\tau_{{\bf{A}}}) \) is the Priestley space dual to the bounded distributive lattice reduct of $\bf{A}$.

\begin{lemma} \label{Objetos PSW-PSWS} If \( \bf{A} \) be a WHB-algebra then the structure \( \X(\bf{A}) \) is a WHBS-space.
\end{lemma}
\begin{proof}
Let \( \bf{A} \) be a WHB-algebra. It follows from the Priestley duality for bounded distributive lattices that \( (X({\bf{A}}),\tau_{{\bf{A}}},\subseteq) \) is a Priestley space. Also by Lemma \ref{marco asociado al algebra} we have that the structure \( (X({\bf{A}}),\subseteq,R_{{\bf{A}}},S_{{\bf{A}}}) \) is a WHB-frame.

Now let us see that the set \( R_{{\bf{A}}}(P) \) and \( S_{{\bf{A}}}(P) \) are \( \tau_{{\bf{A}}} \)-closed for every \(  P \in X({\bf{A}}) \). The proof of \( R_{{\bf{A}}}(P) \) is closed is given in \cite[Theorem 4.12]{CJ1}, so we only show that \( S_{{\bf{A}}} \) is closed. In fact, suppose that \( (P,Q) \notin S_{{\bf{A}}} \) then there are \( a,b \in A \) such that \( a \in Q \), \( b \notin Q \) and \( a \ot b \notin P \), so we have that \( Q \in \sigma_{{\bf{A}}}(a) \setminus \sigma_{{\bf{A}}}(b) \). Note that \( S_{{\bf{A}}}(P) \cap (\sigma_{{\bf{A}}}(a) \setminus \sigma_{{\bf{A}}}(b)) = \emptyset \). In the contrary case, there is \( Z \in S_{{\bf{A}}}(P) \) with \( a \in Z \) and \( b \notin Z \), so we have that \( a \ot b \in P \) which is a contradiction. Thus, \( S_{{\bf{A}}} \) is a \( \tau_{{\bf{A}}} \)-closed set.

Finally we shall see that for each \( W \) clopen set we have that \( R_{{\bf{A}}}^{-1}(W) \) and \( S_{{\bf{A}}}^{-1}(W) \) are clopen sets. Indeed, suppose that \( W \) is \( \tau_{{\bf{A}}} \)-clopen then there is \( a,b \in A \) such that \( W = \sigma_{{\bf{A}}}(a) \setminus \sigma_{{\bf{A}}}(b) \). Note that \( X({\bf{A}}) \setminus R_{{\bf{A}}}^{-1}(\sigma_{{\bf{A}}}(a) \setminus \sigma_{{\bf{A}}}(b)) = \sigma_{{\bf{A}}}(a) \Rightarrow_{R_{{\bf{A}}}} \sigma_{{\bf{A}}}(b)  =  \sigma_{{\bf{A}}}(a \to b) \) and \( S_{{\bf{A}}}^{-1}(\sigma_{{\bf{A}}}(a) \setminus \sigma_{{\bf{A}}}(b)) = \sigma_{{\bf{A}}}(a) \oT_{S_{{\bf{A}}}} \sigma_{{\bf{A}}}(b) = \sigma_{{\bf{A}}}(a \ot b) \) which are a clopen. Thus, we have that the structure \( (X({\bf{A}}),\tau_{{\bf{A}}},\subseteq,R_{{\bf{A}}},S_{{\bf{A}}}) \) is a WHB-space. 
\end{proof}

\begin{lemma} \label{flechas PSW - PSWS} If \( f \colon {\bf{A}}_1 \to {\bf{A}}_2 \) is a homomorphism of WHB-algebras then the map \( X(f) \colon X({\bf{A}}_2) \to X({\bf{A}}_1) \) defined by \( X(f)(P)=f^{-1}(P) \) is a morphism of WHB-spaces.
\end{lemma}
\begin{proof}
Let \( f \colon {\bf{A}}_1 \to {\bf{A}}_2 \) be a homomorphism of WHB-algebras. By the Priestley style duality develop for WH-algebras it follows that \( X(f) \) is a continuous and order preserving map. Besides the conditions (2) and (3) of definition of WHBS-morphism are valid for \( R_{{\bf{A}}_{1}} \) and \( R_{{\bf{A}}_{2}} \). So it will be enough prove the conditions (2) and (3) of the definition of WHBS-morphism for \( S_{{\bf{A}}_{1}} \) and \( S_{{\bf{A}}_{2}} \). 

Let \( P,Q \in X({\bf{A}}_2) \) with \( (P,Q) \in S_{{\bf{A}}_{2}} \) we shall see that \( (f^{-1}(P),f^{-1}(Q)) \in S_{{\bf{A}}_{1}} \). In fact, let \( a \in f^{-1}(Q) \) and \( b \notin f^{-1}(Q) \), so we have that \( f(a) \in Q \) and \( f(b) \notin Q \). Since \( (P,Q) \in S_{{\bf{A}}_{2}} \) we have that \( f(a) \ot f(b) \in P \), so \( f(a \ot b) \in P \) which implies that \( a \ot b \in f^{-1}(P) \) and \( (f^{-1}(P),f^{-1}(Q)) \in S_{{\bf{A}}_{1}} \).

Now, suppose that \( (f^{-1}(P),Z) \in S_{{\bf{A}}_{1}} \) we shall prove that there is \( Q \in X({\bf{A}}_{2}) \) such that \( (P,Q) \in S_{{\bf{A}}_{2}} \) and \( f^{-1}(Q) = Z \). In order to prove it, we prove that 
\[ F_{P}(f(Z)) \cap (f(Z)^{c}] = \emptyset \]
In fact, suppose that \( x \in F_{P}(f(Z)) \cap (f(Z)^{c}] \) then there is \( z \in Z \) and \( u \in Z^{c} \) such that \( f(z) \ot x \notin P \) and \( x \leq f(u) \), so we have that \( f(z) \ot f(u) \notin P \) and \( z \ot u \notin f^{-1}(P) \). Since \( (f^{-1}(P),Z) \in S_{{\bf{A}}_{1}} \) we have that \( z \ot u \in f^{-1}(P) \) which is a contradiction. Thus, by Theorem \ref{DTPF} there is \( Q \in X({\bf{A}}_{2}) \) such that \( (P,Q) \in S_{{\bf{A}}_{2}} \) and \( f^{-1}(Q) = Z \). \end{proof}

Straightfoward computations show that we have defined a contravariant functor \( \X \colon \mathsf{WHB} \to \mathsf{WHBS} \) whose action in objects and arrows is described in lemmas \ref{Objetos PSW-PSWS} and \ref{flechas PSW - PSWS}.\\

\vspace{1pt}

Let \( \sigma \colon \mathrm{I}_{\mathsf{WHB}} \to \D \circ \X \) be the natural transformations such that for every \( \mathbf{A} \) object of \( \mathsf{WHB} \) the arrow \( \sigma_{{\bf{A}}} \) is the Stone map and \( \epsilon \colon \mathrm{I}_{\mathsf{WHBS}} \to \X \circ \D \) the natural transformation such that for every \( (X,\tau,\leq,R,S) \) object of \( \mathsf{WHBS} \) the arrow \( \epsilon_{X} \) is defined by \( \epsilon_{X}(x) = \{ U \in \D(X) \colon x \in U \} \). It follows from Theorem \ref{representacion WHB} that the map \( \sigma_{{\bf{A}}} \) is an arrow in \( \mathsf{WHB} \). Moreover the map \( \sigma_{\bf{A}} \) is a isomorphism between \( \bf{A} \) and \( \D(\X({\bf{A}})) \). Our next goal is to show that the map $\epsilon_{X}$ is an arrow of the category $\mathsf{WHBS}$.

\begin{lemma} \label{aux 3} If \( (X,\tau,\leq,R,S) \) be an object in \( \mathsf{WHBS} \) then the map \( \epsilon_{X} \) is an arrow in \( \mathsf{WHBS} \).  
\end{lemma}
\begin{proof}
Let \( (X,\tau,\leq,S) \) be an object in \( \mathsf{WHBS} \). In order to prove that \( \epsilon_{X} \) is an arrow in \( \mathsf{WHBS} \) we shall prove that \( (x,y) \in R \) if and only if \( (\epsilon_{X}(x),\epsilon_{X}(y)) \in R_{\D(X)} \) and \( (x,y) \in S \) if and only if \( (\epsilon_{X}(x),\epsilon_{X}(y)) \in S_{\D(X)} \). The fact that $(x.y) \in R$ if and only if $\left( \epsilon_{X}(x),\epsilon_{X}(y) \right) \in R_{\D(X)}$ it follows from \cite[Proposition 4.11]{CJ1}, so we only show that $(x,y) \in S$ if and only if $\left( \epsilon_{X}(x),\epsilon_{X}(y) \right)$.

Now suppose that \( (x,y) \in S \) we shall prove that \( (\epsilon_{X}(x),\epsilon_{X}(y)) \in S_{\D(X)} \). Let \( U \in \epsilon_{X}(y) \) and \( V \notin \epsilon_{X}(y) \), so \( y \in U \setminus V \) and \( (x,y) \in S \) which implies that \( S(x) \cap (U \setminus V) \neq \emptyset \) and \( x \in U \oT_{S} V \) i.e, \( U \oT_{S} V \in \epsilon(x) \). Conversely, suppose that \( (x,y) \notin S \). Since \( S(x) \) is \( \tau \)-closed there are \( U,V \in D(X) \) such that \( S(x) \cap U \setminus V = \emptyset \) and \( y \in U \setminus V \). Thus, \( U \in \epsilon_{X}(y) \), \( V \notin \epsilon_{X}(y) \) and \( U \oT_{S} V \notin \epsilon_{X}(x) \), which implies that \( (\epsilon_{X}(x),\epsilon_{X}(y)) \notin S_{\D(X)} \). 
\end{proof}

It follows from Priestley duality for bounded distributive lattices and Lemma \ref{aux 3} that the map \( \epsilon_{X} \colon X \to \X(\D(X)) \) is an isomorphism in the category $\mathsf{WHBS}$ for any $X$ object in \( \mathsf{WHBS} \). Thus we have the following statement.

\begin{theorem}\label{Duality WHB} The category \( \mathsf{WHB} \) is dually equivalent with the category \( \mathsf{WHBS} \) with natural isomorphism given by the maps \( \sigma \) and \( \epsilon \) respectively.
\end{theorem}

\subsection{The dual of congruences} \label{S5.1}

In this subsection we will use the Priestley duality for $\mathsf{WHB}$ in order to characterice the lattice $\mathrm{Con}(\mathbf{A})$ for every WHB-algebra $\mathbf{A}$. We start this section by giving some definitions that will be necessary.

\begin{definition} Let $(X,\tau,\leq,R,S)$ be a WHB-space and $Y \subseteq X$. We say that:
\begin{enumerate}[\normalfont 1.]
    \item $Y$ is closed by the relation $\mathcal{R}$ if $x \in Y$ implies $\mathcal{R}(x) \subseteq Y$, for $\mathcal{R} \in \{R,S\}$
    \item $Y$ is doubly closed if it closed is by the relations $R$ and $S$. 
\end{enumerate}
\end{definition}
We write $C_{R}(X)$ to denote the lattice of closed sets and $R$-closed. Similarly we write $C_{S}(X)$ and $DC(X)$ to denote the lattice  of closed and $S$ closed sets and doubly closed sets respectively. 

\begin{definition} Let $\mathbf{A}$ be a WHB-algebra, $\theta \in \mathrm{Con}(\mathbf{A})$ and $P \in X(\mathbf{A})$. We say that $\theta$ is compatible with $P$ whenever $(a,b) \in \theta$ and $a \in P$ implies that $b \in P$
\end{definition}

It is known that if $\mathbf{L}$ is a bounded distributive lattice then there is a lattice isomorphism between the lattice of congruences $\mathrm{Con}(\mathbf{L})$ and the lattice $C(X(\mathbf{L}))$ whose elements are closed subsets of its dual Priestley space $(X(\mathbf{L}),\tau_{\mathbf{L}},\subseteq)$. Moreover this isomorphism is given by the maps $\Theta \colon C(X(\mathbf{L})) \to \mathrm{Con}(\mathbf{L})$ defined by 
\[\Theta(Y)= \{ (a,b) \in L \times L \colon \sigma_{\mathbf{L}} \cap Y = \sigma_{L}(b) \cap Y \}, \] 
and its inverse map $\Theta^{-1} \colon \mathrm{Con}(\mathbf{L}) \to C(X({\bf{L}}))$ defined by 
\[ \Theta^{-1}(\theta) = \{ P \in X(\mathbf{L}) \colon \theta \text{ is compatible  with } P \} \]

In \cite[Theorem 6.1]{CJ2} it was proved that for every WH-algebra $\mathbf{A}$ the maps $\Theta$ and $\Theta^{-1}$ give an isomorphism between the lattice $\mathrm{Con}(\mathbf{A})$ and the lattice $C_{R_{{\bf{A}}}}(X({\bf{A}}))$. So, it follows the following result

\begin{theorem} \label{iso congruencias y cerrados} Let $\mathbf{A}$ be a WHB-algebra and $(X(\mathbf{A}),\tau_{{\bf{A}}},\subseteq,R_{{\bf{A}}},S_{{\bf{A}}})$ its dual WHB-space. Then the maps $\Theta$ and $\Theta^{-1}$ give a lattice isomorphism between the lattice $\mathrm{Con}(\mathbf{A})$ and the lattice $DC(X({\bf{A}}))$.
\end{theorem}
\begin{proof}
Taking into account \cite[Theorem 6.1]{CJ2} we only need to show that if $Y \in C_{S_{\mathbf{A}}}(X(\mathbf{A}))$ then for each $(a,b), (c,d) \in  \Theta(Y)$ we have that $(a \ot c, b \ot d) \in \Theta(Y)$ and if $\theta \in \mathrm{Con}(\mathbf{A})$ then $Y(\theta)$ is closed by the relation $S_{\mathbf{A}}$. 

Suppose that $Y \in C_{S_{\mathbf{A}}}(X(\mathbf{A}))$ and let $(a,b),(c,d) \in \Theta(Y)$ we shall prove that $(a \ot c,b \ot d) \in \Theta(Y)$. In fact, let $P \in \sigma_{\mathbf{A}}(a \ot c) \cap Y$ then by Corollary \ref{corolario DTFP} there is $Q \in X(\mathbf{A})$ such that $(P,Q) \in S_{\mathbf{A}}$, $a \in Q$ and $c \notin Q$. Since $Y$ is closed by the relation $S_{\mathbf{A}}$ we have that $Q \in Y \cap \sigma_{\mathbf{A}}(a) = Y \cap \sigma_{\mathbf{A}}(b)$. Also notice that $d \notin Q$. Indeed, if $d \in Q$ we have that $Q \in Y \cap \sigma_{\mathbf{A}}(d) = Y \cap \sigma_{\mathbf{A}}(c)$, so $c \in Q$ which is absurd. Thus, $(P,Q) \in S_{\mathbf{A}}$, $b \in Q$ and $d \notin Q$ which implies that $P \in Y \cap \sigma_{\mathbf{A}}(b \ot d)$. The proof of the remaining inclusion is analogue. Thus, $(a \ot c,b \ot d) \in \Theta(Y)$ which was our aim. 

Conversely, suppose that $\theta \in \mathrm{Con}(\mathbf{A})$. We shall prove that $\Theta^{-1}(\theta)$ is closed by the relation $S_{\mathbf{A}}$. In fact, let $P \in \Theta^{-1}(\theta)$ and $Q \in X(\mathbf{A})$ such that $(P,Q) \in S_{\mathbf{A}}$. We shall prove that $\theta$ is compatible with $Q$. Suppose the contrary case. Then there is $(a,b) \in \theta$ such that $a \in Q$ and $b \notin Q$. Taking into account that $(P,Q) \in S_{\mathbf{A}}$ we have that $a \ot b \in P$. Also, since $(a,b) \in \theta$ we have that $(a \ot b,0) \in \theta$ and $0 \notin P$, so by the compatibility of $P$ we have that $a \ot b \notin P$ which is a contradiction. Thus, $\theta$ is compatible with $Q$.
\end{proof}

\section{The free tense extension of a strict weak algebra} \label{S6}

We start this section by recalling that given posets $P$ and $Q$, a pair of functions $f:P\rightarrow Q$ and $g:Q\rightarrow P$ is said to be an \emph{adjoint pair} (or Galois connection in \cite{DP2002}) if the following condition holds for every $p\in P$ and $q\in Q$:
\begin{equation}
\begin{array}{ccc}
f(p)\leq q & \Leftrightarrow & p\leq g(q).
\end{array}
\end{equation}
In such a case we say that $f$ is \emph{left adjoint} to $g$ or equivalently, that $g$ is \emph{right adjoint} to $f$. In the literature this situation is usually denoted by $f \dashv g$. For a deeper treatment on adjoint pairs between posets the reader can see Chapter 7 of \cite{DP2002}. A \emph{tense Boolean algebra} (tense algebra, for short), is a structure $(\mathbf{B}, G, H)$ such that $\mathbf{B}$ is a Boolean algebra and $G$ and $H$ are unary operators on $B$ such that $P\dashv G$ and $F\dashv H$, with $P(x):=\neg H(\neg x)$  and $F(x):=\neg G(\neg x)$. It is well known that tense algebras form a variety denoted by $\mathcal{TBA}$. An equational presentation of these algebras can be found in \cite{K1998}. Tense algebras where introduced as an algebraic semantic for Lemon's tense logic (see \cite{B1984} and all the references therein). We stress that every tense algebra can be regarded as a WHB-algebra, by defining $x\rightarrow y:=G(\neg x\vee y)$ and $x\leftarrow y:=P(x\wedge \neg y)$. Moreover, if we present tense algebras as structures $(\mathbf{B},\rightarrow,\leftarrow)$, then every homomorphism of tense algebras is in fact a homomorphism of WHB-algebras. So, if $\mathsf{TBA}$ denotes the category of tense Boolean algebras and homomorphisms, then it follows that there is a faithful functor $M:\mathsf{TBA}\rightarrow \mathsf{WHB}$. Furthermore, we claim that $M$ restricts to an isomorphism of categories when considering the full subcategory $\mathsf{WHB}^{\ast}$ of $\mathsf{WHB}$ whose objects are those WHB-algebras whose lattice reduct is Boolean. Observe that the inverse of $M$ sends an object $(\mathbf{B}, \rightarrow, \leftarrow)$ of $\mathsf{WHB}^{\ast}$ into the tense algebra $(\mathbf{B},G, H)$, where $H(x)= \neg (\neg x\leftarrow 0)$ and $G(x)=1\rightarrow x$.  
\\

In this section we will prove that there exists a functor $T:\mathsf{WHB}\rightarrow \mathsf{TBA}$ which is left adjoint to $M$. Such a functor will be build by employing Theorem \ref{Duality WHB} and the well known free Boolean extension of a bounded distributive lattice. If $\mathbf{A}$ is a WHB-algebra, then we will refer to $T(\mathbf{A})$ as the \emph{free tense extension} of $\mathbf{A}$. We also will show that the free tense extension can be very useful for studying structural properties of subvarieties of WHB-algebras in terms of tense algebras. 
\\

Now we recall some facts about the free Boolean extension of a bounded distributive lattice. Let $\mathbf{L}$ be a bounded distributive lattice and let $\sigma_{\mathbf{L}}: {\bf{L}} \rightarrow \mathcal{P}(X({\bf{L}}))$ be the Stone map. Consider the set 
\[\mathcal{B}(L)=\{\bigcup_{i\leq n}(\sigma_{\mathbf{L}}(a_{i}) \backslash\sigma_{{\mathbf{L}}}(b_{i}))\colon a_{i},b_{i} \in L, n\in \mathbb{N}\}.\]
It is well known that the structure $\mathcal{B}(\mathbf{L})=(\mathcal{B}(L),\cup,\cap, ^{c},X(L), \emptyset)$ is a Boolean algebra. If $h:\mathbf{L}\rightarrow \mathbf{M}$ is a lattice homomorphism, then from the commutativity of the diagram below
\begin{displaymath}
\xymatrix{
\mathbf{L} \ar[r]^-{\sigma_{\mathbf{L}}} \ar[d]_-{h} & \mathcal{P}(X(\mathbf{L})) \ar[d]^-{X(h)^{-1}}
\\
\mathbf{M} \ar[r]_-{\sigma_{\mathbf{M}}} & \mathcal{P}(X(\mathbf{M}))
}
\end{displaymath}

we get that the map $\mathcal{B}(h)= X(h)^{-1}$ is a well defined lattice homomorphism from $\mathcal{B}(\mathbf{L})$ to $\mathcal{B}(\mathbf{M})$. Then, the assignments $\mathbf{L}\mapsto \mathcal{B}(\mathbf{L})$ and $h\mapsto \mathcal{B}(h)= X(h)^{-1}$ determine a functor $\mathcal{B}$ from the category of distributive bounded lattices $\mathsf{DLat}$ to the category of Boolean algebras $\mathsf{Boole}$ which is a left adjoint to the forgetful functor $U:\mathsf{Boole} \rightarrow \mathsf{DLat}$. The Boolean algebra  $\mathcal{B}(\mathbf{L})$ is known as the \emph{free Boolean extension} of $\mathbf{L}$ and such an algebra coincides with the Boolean algebra of clopen upsets of the Priestley dual space of $\mathbf{L}$. We conclude by stressing that the employment of the terms ``free'' and ``extension'' are justified by the fact that the unit of the adjunction $\sigma_{\mathbf{L}}: \mathbf{L} \rightarrow U\mathcal{B}(\mathbf{L})$ is a lattice embedding.
\\

Let \( (X,\tau,\leq,R,S) \) be a WHB-space and let $B(X)$ be the set of clopen subsets of $X$. From Definition \ref{def WHB-space} (3), the maps $P$ and $F$, defined by $P(U)=S^{-1}(U)$ and $F(U)=R^{-1}(U)$ form a pair of well defined operators on $B(X)$. The latter immediately implies that the maps $G$ and $H$ defined by $G(U)=X\backslash F(X\backslash U)$ and $H(U)=X\backslash P(X\backslash U)$ are also well defined operators on $B(X)$.

Let $\mathbf{A}$ be a WHB-algebra. In what follows, we denote by $\mathcal{B}({\bf{A}})$ the set of clopen subsets of the dual WHB-space of $\mathbf{A}$. Notice that the elements of $\mathcal{B}({\bf{A}})$ coincide with the elements of the free Boolean extension of the lattice reduct of $\mathbf{A}$. Moreover, from Theorem \ref{Objetos PSW-PSWS}, it is also the case that $G_{{\bf{A}}}(U)= X({\bf{A}}) \Rightarrow_{R_{{\bf{A}}}}U$, $P_{{\bf{A}}}(U)=U\Leftarrow_{S_{{\bf{A}}}}\emptyset$, $H_{{\bf{A}}}(U)=X({\bf{A}})\backslash ((X({\bf{A}})\backslash U) \Leftarrow_{S_{{\bf{A}}}} \emptyset)$ and $F_{{\bf{A}}}(U)=X({\bf{A}})\backslash (X({\bf{A}})\Rightarrow_{R_{{\bf{A}}}} (X({\bf{A}})\backslash U))$, for every $U\in \mathcal{B}({\bf{A}})$.

\begin{lemma}\label{B(A) es tense algebra} 
Let $\mathbf{A}$ be a WHB-algebra. Then, the structure \[T(\mathbf{A})=(\mathcal{B}({\bf{A}}), \cup, \cap, ^{c},G_{{\bf{A}}},H_{{\bf{A}}},\emptyset, X({\bf{A}})),\] is a tense algebra.
\end{lemma}
\begin{proof}
In order to prove the statement we need to check that the following conditions hold, for every $U\in \mathcal{B}(A)$:
\begin{enumerate}
\item $F_{{\bf{A}}}H_{{\bf{A}}}(U)\subseteq U\subseteq H_{{\bf{A}}}F_{{\bf{A}}}(U)$,
\item $P_{{\bf{A}}}G_{{\bf{A}}}(U)\subseteq U\subseteq G_{{\bf{A}}}P_{{\bf{A}}}(U)$.
\end{enumerate}

We only prove (1) because the proof of (2) is similar. We start by checking the inclusion $F_{{\bf{A}}}H_{{\bf{A}}}(U)\subseteq U$. So, let $P\in X({\bf{A}})$ and suppose that $P\in F_{\bf{A}}(H_{{\bf{A}}}(U))$, then there exists $Q\in X(\bf{A})$ such that  $(P,Q)\in R_{\bf{A}}$ and $S_{{\bf{A}}}(Q)\subseteq U$. From Lemma \ref{ecuaciones y relaciones}, $(Q,P)\in S_{{\bf{A}}}$ thus $P\in U$, as claimed. Now we prove $U\subseteq H_{{\bf{A}}}F_{{\bf{A}}}(U)$. To do so, suppose that there exists $P\in X({\bf{A}})$ such that $P\in U$ but $P\notin H_{{\bf{A}}}F_{{\bf{A}}}(U)$, so we have that $P \in \left( X({\bf{A}}) \setminus F_{{\bf{A}}}(U) \right) \oT_{S_{{\bf{A}}}} \emptyset$. Then, there exists $Q\in X({\bf{A}})$ such that $(P,Q)\in S_{{\bf{A}}}$ and $Q \in X({\bf{A}}) \To_{R_{{\bf{A}}}} X({\bf{A}}) \setminus U$ which implies that $R_{{\bf{A}}}(Q) \subseteq X({\bf{A}}) \setminus U$. Therefore, from the latter and Lemma \ref{ecuaciones y relaciones} we get $(Q,P)\in R_{{\bf{A}}}$, so $P \in X({\bf{A}}) \setminus U$ which is absurd. This completes the proof.  

\end{proof}

\begin{remark}\label{B(A) no es temporal S4}
Let $\bf{A}$ be a WHB-algebra and $\mathcal{F}({\bf{A}})= (X({\bf{A}}),\subseteq,R_{{\bf{A}}},S_{{\bf{A}}})$ its frame associated. By lemmas \ref{aux1} and \ref{aux2} we have that  the conditions $\rm{R}$ and $\rm{T}$ are satisfied on $\bf{A}$ if and only if $R_{\bf{A}}$ is reflexive and transitive. Moreover, $R_{\bf{A}}$ is reflexive and transitive if and only if the equations $\rm{R}$ and $\rm{T}$ are valid on $\mathcal{A}(\mathcal{F}({\bf{A}}))$. In particular, we have that the operator $G_{\bf{A}}$ satisfy the following conditions:
\begin{enumerate}[\normalfont (1)]
    \item $G_{{\bf{A}}}(U) \subseteq U$
    \item $G_{{\bf{A}}}(U) \subseteq G_{{\bf{A}}}(G_{{\bf{A}}}(U))$
\end{enumerate}
Dually, taking into account lemmas \ref{ecuacion y condicion WD} and \ref{lema condicion ecuacion WD} we have that the conditions ${\rm{R}}^{*}$ and ${\rm{T}}^{*}$ are valid on $\bf{A}$ if and only if $S_{{\bf{A}}}$ is reflexive and transitive. Moreover, we have that $S_{{\bf{A}}}$ is reflexive and transitive if and only if the equation ${\rm{R}}^{*}$ and ${\rm{T}}^{*}$ are valid on $\mathcal{A}(\mathcal{F}(\mathbf{A}))$. In particular, it follows that the operator $H_{{\bf{A}}}$ satisfies the following equations:
\begin{enumerate}[\normalfont (1)]
    \item $H_{{\bf{A}}}(U) \subseteq U$
    \item $H_{{\bf{A}}}(U) \subseteq H_{{\bf{A}}}(H_{{\bf{A}}}(U))$
\end{enumerate}
In particular, if ${\bf{A}}$ is a HB-algebra then the structure $T(\mathbf{A})$ defined in Lemma \ref{B(A) es tense algebra} is a $\mathbf{S}_{4}$ Tense algebra. This result is also given in \cite{Wolter}.
\end{remark}

Let $h:\mathbf{A}\rightarrow \mathbf{B}$ be a homomorphism of WHB-algebras. We recall that from Lemma \ref{flechas PSW - PSWS}, the map $\mathbf{X}(h):\mathbf{X}(\mathbf{B})\rightarrow \mathbf{X}(\mathbf{A})$ is a morphism of WHB-spaces. Moreover, because of Theorem \ref{Duality WHB}, for every $a\in A$ we have $\mathbf{X}(h)^{-1}(\sigma_{\mathbf{A}}(a))=\sigma_{\mathbf{B}}(h(a))$. Observe that these facts imply that the map $T(h):T(\mathbf{A})\rightarrow T(\mathbf{B})$, defined by 
\[T(h)(\bigcup_{i\leq n}(\sigma_{\mathbf{A}}(a_i)\backslash \sigma_{\mathbf{A}}(b_i))=\bigcup_{i\leq n}(\sigma_{\mathbf{B}}(h(a_i))\backslash \sigma_{\mathbf{B}}(h(b_i))\]  
is a well defined homomorphism of tense algebras. It is no hard to see that the assignments $\mathbf{A}\mapsto T(\mathbf{A})$, $h\mapsto T(h)$, determine a functor $T:\mathsf{WHB}\rightarrow \mathsf{TBA}$.
\\

Let $F:\mathsf{C}\rightarrow \mathsf{D}$ and $G:\mathsf{D}\rightarrow \mathsf{C}$ be functors. We recall that $F$ is \emph{left adjoint to} $G$ if and only if, for every object $A$ in $\mathsf{C}$ and $B$ in $\mathsf{D}$, there is a bijection between the arrows $F(A)\rightarrow B$ in $\mathsf{D}$ and the arrows $A\rightarrow G(B)$ in $\mathsf{C}$. It is well known (Theorem 4.2 of \cite{MCL2013}) that this is equivalent to the existence of a natural transformation $\eta:Id_{\mathsf{C}} \Rightarrow GF$ which is universal for every object $A$ in $\mathsf{C}$. I.e. for every arrow $f: A\rightarrow G(B)$ in $\mathsf{C}$, there exists a unique arrow $h:F(A)\rightarrow B$ in $\mathsf{D}$ such that $f=G(h)\eta_{A}$. In such a case, the natural transformation $\eta$ is usually called the \emph{unit} of the adjunction. Notice that when considering posets as categories, the notions of adjoint pair and adjoint functors coincide.

\begin{theorem}\label{la extension temporal}
The functor $T$ is left adjoint to $M$. Moreover, $T$ preserves monomorphisms, it is faithful and $M$ is full and faithful.
\end{theorem}
\begin{proof}
It is no hard to see that that the collection of maps $\sigma_{\mathbf{A}}: \mathbf{A}\rightarrow MT(\mathbf{A})$, with $\mathbf{A}$ object of $\mathsf{WHB}$, induce a natural transformation $\sigma:Id_{\mathsf{WHB}} \Rightarrow MT$. Now we prove that $\sigma_{\mathbf{A}}$ is universal for every WHB-algebra $\mathbf{A}$. To do so, let $\mathbf{B}$ be a tense algebra and $h:\mathbf{A}\rightarrow \mathbf{B}$ be a homomorphism of WHB-algebras. Then, from Theorem \ref{Duality WHB} it follows that $MT(h): MT(\mathbf{A})\rightarrow MT(\mathbf{B})$ is a homomorphism of WHB-algebras making the following diagram 
\begin{displaymath}
\xymatrix{
\mathbf{A} \ar[r]^-{\sigma_{\mathbf{A}}} \ar[d]_-{h} & MT(\mathbf{A}) \ar[d]^-{MT(h)}
\\
\mathbf{B} \ar[r]_-{\sigma_{\mathbf{B}}} & MT(\mathbf{B})
}
\end{displaymath}
commutes. Since $\sigma_{\mathbf{B}}$ is an isomorphism, we can consider $\psi=\sigma_{\mathbf{B}}^{-1}MT(h)$. From the commutativity of the diagram above it follows that $\psi$ is the only homomorphism of WHB-algebras such that $h=\psi \sigma_{\mathbf{A}}$. Thus $\sigma_{\mathbf{A}}$ is universal, as claimed. On the other hand, if $h:\mathbf{A}\rightarrow \mathbf{B}$ is mono in $\mathsf{WHB}$, then from Theorem \ref{Duality WHB} $X(h):X(\mathbf{B})\rightarrow X(\mathbf{A})$ is an epimorphism of WHB-spaces, so $T(h)=X(h)^{-1}$ is mono in $\mathsf{TBA}$. We also stress that the faithfulness of $T$ also follows from Theorem \ref{Duality WHB}. Finally, $M$ is full and faithful because the counit of the adjunction is an isomorphism in $\mathsf{TBA}$.
\end{proof}

From Theorem \ref{la extension temporal} and the isomorphism between $\mathsf{WHB}^{\ast}$ and $\mathsf{TBA}$, we get the following.

\begin{corollary}\label{subcategory of WHB}
$\mathsf{WHB}^{\ast}$ is a reflective subcategory of $\mathsf{WHB}$.
\end{corollary}

It is well know that given a bounded distributive lattice $\mathbf{L}$, when applying the Priestley's duality functor $X$ to the unit $\sigma_{\mathbf{L}}: \mathbf{L} \rightarrow U\mathcal{B}(\mathbf{L})$ we get that $X(\sigma_{\mathbf{L}})$ becomes an isomorphism between the Stone spaces $(X(\mathbf{L}), \tau_{\mathbf{L}})$ and $X(U\mathcal{B}(\mathbf{L}))$ (or equivalently, the Priestley spaces $(X(\mathbf{L}), =, \tau_{\mathbf{L}})$ and $X(U\mathcal{B}(\mathbf{L}))$). Notice that this situation in some sense repeats when considering the categories $\mathsf{WHB}$ and $\mathsf{TBA}$ via the free tense extension of a WHB-algebra. In order to explain how, let $\mathbf{A}$ be a WHB-algebra and let $\sigma_{\mathbf{A}}:\mathbf{A}\rightarrow MT(\mathbf{A})$ be the unit of the adjunction obtained in Theorem \ref{la extension temporal}. Then, from Theorem \ref{Duality WHB}, the map $X(\sigma_{\mathbf{A}}): X(MT(\mathbf{A}))\rightarrow X(\mathbf{A})$ is a morphism of WHB-spaces. Observe that the same arguments employed for case of bounded distributive lattices work for proving that $X(\sigma_{\mathbf{A}})$ is in fact an isomorphism between the WHB-spaces $(X({\bf{A}}),\tau_{\mathbf{A}},=,R_{{\bf{A}}},S_{{\bf{A}}})$ and $X(MT(\mathbf{A}))$. Hence, we have shown:

\begin{lemma}\label{free tense extension gives the same space}
For every WHB-algebra $\mathbf{A}$, the WHB-spaces $(X({\bf{A}}),\tau_{\mathbf{A}},=,R_{{\bf{A}}},S_{{\bf{A}}})$ and $X(MT(\mathbf{A}))$ are isomorphic.
\end{lemma}

\begin{theorem}\label{Con(A) iso Con(T(A))}
For every WHB-algebra $\mathbf{A}$, there is an isomorphism between the lattices $\mathsf{Con}_{WHB}(\mathbf{A})$ and $\mathsf{Con}_{TBA}(T(\mathbf{A}))$.
\end{theorem}
\begin{proof}
Suppose that $\mathbf{A}$ is a WHB-algebra. Note that the WHB-spaces $X(\mathbf{A})$ and $(X({\bf{A}}),\tau_{\mathbf{A}},=,R_{{\bf{A}}},S_{{\bf{A}}})$ has the same set of doubly closed sets. Therefore we get the following chain of isomorphisms:
\begin{displaymath}
\begin{array}{ccll}
\mathsf{Con}_{\mathcal{WHB}}(\mathbf{A}) & \cong & DC(X(\mathbf{A}))^{op} & \text{by Theorem}\; \ref{iso congruencias y cerrados},
\\
 & \cong &  DC((X(\mathbf{A}),\tau_{\mathbf{A}},=,R_{A},S_{A})))^{op} & 
\\
 & \cong &  DC(X(MT(\mathbf{A})))^{op} & \text{by Lemma}\; \ref{free tense extension gives the same space}, 
\\
 & \cong &  \mathsf{Con}_{\mathcal{WHB}}(MT(\mathbf{A}))  & \text{by Theorem}\; \ref{iso congruencias y cerrados},
\\
 & \cong &  \mathsf{Con}_{\mathcal{TBA}}(T(\mathbf{A})).  & 
\end{array}
\end{displaymath}
The last isomorphism comes from the fact that a Boolean algebra and its lattice reduct have the same congruences. This concludes the proof.
\end{proof}

Let $\mathbf{A}$ be a WHB-algebra. In what follows, we denote by $\Phi$ the isomorphism of Theorem \ref{Con(A) iso Con(T(A))}.

\begin{theorem}
The variety $\mathcal{WHB}$, is arithmetical. 
\end{theorem}
\begin{proof}
Let $\mathbf{A}$ be a WHB-algebra. Since, $\mathcal{TBA}$ is congruence-distributive (\cite{K1998}), the lattice $\mathsf{Con}_{\mathcal{TBA}}(T(\mathbf{A}))$ is distributive, thus from Theorem \ref{Con(A) iso Con(T(A))}, we get that $\mathsf{Con}_{\mathcal{WHB}}(\mathbf{A})$ is distributive. Finally, let $\theta_{1}, \theta_{2}\in \mathsf{Con}_{\mathcal{WHB}}(\mathbf{A})$. Observe that again by Theorem \ref{Con(A) iso Con(T(A))}, it is no hard to see that $\Phi(\theta_{1}\circ \theta_{2})=\Phi(\theta_{1})\circ \Phi(\theta_{2})$. Therefore, by the assumption on $\mathsf{Con}_{\mathcal{TBA}}(T(\mathbf{A}))$, we get $\Phi(\theta_{1}\circ \theta_{2})=\Phi(\theta_{2}\circ \theta_{1})$, so $\theta_{1}\circ \theta_{2}=\theta_{2}\circ \theta_{1}$, as desired.
\end{proof}

\begin{corollary}\label{lema EDPC}
Let $\mathbf{A}$ be a WHB-algebra and $\mathbf{B}$ be a tense algebra. Then, the following holds:
\begin{enumerate}[\normalfont (1)]
\item For every $a,b\in A$, $\Phi(\mathsf{Cg}^{\mathbf{A}}(a,b))=\mathsf{Cg}^{T(\mathbf{A})}(\sigma_{\mathbf{A}}(a),\sigma_{\mathbf{A}}(b))$.
\item For every $p,q\in B$, $\mathsf{Cg}^{\mathbf{B}}(p,q)=\mathsf{Cg}^{M(\mathbf{B})}(p,q)$.
\end{enumerate}
\end{corollary}
\begin{proof}
(1) We start by recalling that \[\mathsf{Cg}^{\mathbf{A}}(a,b)=\bigcap \{\theta \in \mathsf{Con}_{\mathcal{WHB}}(\mathbf{A})\colon (a,b)\in \theta\}.\]
Thus, from Theorem \ref{Con(A) iso Con(T(A))} we get that:
\begin{displaymath}
\begin{array}{rcl}
\Phi(\mathsf{Cg}^{\mathbf{A}}(a,b)) & = & \bigcap \{\Phi(\theta) \in \mathsf{Con}_{\mathcal{TBA}}(T(\mathbf{A}))\colon (a,b)\in \theta\}
\\
 & = & \bigcap \{\Phi(\theta) \in \mathsf{Con}_{\mathcal{TBA}}(T(\mathbf{A}))\colon (\sigma_{\mathbf{A}}(a),\sigma_{\mathbf{A}}(b))\in \Phi(\theta)\}
\\ 
 & = & \mathsf{Cg}^{T(\mathbf{A})}(\sigma_{\mathbf{A}}(a),\sigma_{\mathbf{A}}(b)).
\end{array}
\end{displaymath}
(2) We recall that the counit of the adjunction $\varepsilon_{\mathbf{B}}:TM(\mathbf{B})\rightarrow \mathbf{B}$ of Theorem \ref{la extension temporal} is an isomorphism for every $\mathbf{B}\in \mathcal{TBA}$. So, from general reasons, the latter implies that the map $\kappa:\mathsf{Con}_{\mathcal{TBA}}(TM(\mathbf{B}))\rightarrow \mathsf{Con}_{\mathcal{TBA}}(\mathbf{B})$ defined as 
\[\kappa(\theta) = \{(\varepsilon_{\mathbf{B}}(a),\varepsilon_{\mathbf{B}}(b))\colon (a,b)\in \theta\}  \]
is also an isomorphism. Therefore, if $s,t\in B$, by Theorem \ref{Con(A) iso Con(T(A))}, (1) and the application of $\kappa$ on $\mathsf{Cg}^{TM(\mathbf{B})}(\sigma_{M(\mathbf{B})}(p),\sigma_{M(\mathbf{B})}(q))$, it is the case that:
\begin{displaymath}
\begin{array}{rcl}
(s,t)\in \mathsf{Cg}^{M(\mathbf{B})}(p,q) & \Longleftrightarrow & (\sigma_{M(\mathbf{B})}(s),\sigma_{M(\mathbf{B})}(t))\in \mathsf{Cg}^{TM(\mathbf{B})}(\sigma_{M(\mathbf{B})}(p),\sigma_{M(\mathbf{B})}(q))
\\
& \Longleftrightarrow & (s,t)\in \mathsf{Cg}^{\mathbf{B}}(p,q),
\end{array}
\end{displaymath}
as desired.
\end{proof}

Let $\mathbf{B}$ be a tense algebra. A subset $F$ of $B$ is said to be a \emph{tense filter}, provided that it is a lattice filter and if for every $x\in F$, $G(x)$ and $H(x)$ belong to $F$. In Fact 1.11.  of \cite{K1998} it was proved that the lattice of tense filters of $\mathbf{A}$ is isomorphic to $\mathsf{Con}_{\mathcal{TBA}}(\mathbf{B})$. Thus, as an immediate application of Theorem \ref{Con(A) iso Con(T(A))} we obtain the following: 

\begin{corollary}
Let $\mathbf{A}$ be a WHB-algebra. Then, there is an isomorphism between $\mathsf{Con}_{\mathcal{WHB}}(\mathbf{A})$ and the set of tense filters of $T(\mathbf{A})$.
\end{corollary}

Let $X$ be a set of variables. If $\mathcal{L}$ is a propositional language we write $\mathbf{Fm}_{\mathcal{L}}(X)$ for the absolutely free algebra of $\mathcal{L}$-formulas (terms in the language $\mathcal{L}$).  If $\mathcal{V}$ is a variety, we write $\mathbf{F}_{\mathcal{V}}(X)$ for the $\mathcal{V}$-free algebra over $X$. Further, if $\mathcal{L}$ is the language of $\mathcal{V}$ and $\varphi$ is an $\mathcal{L}$-formula, we also denote by $\varphi$ its image under the natural map $\mathbf{Fm}_{\mathcal{L}}(X)  \rightarrow \mathbf{F}_{\mathcal{V}}(X)$ from the term algebra $\mathbf{Fm}_{\mathcal{L}}(X)$ over $X$ onto $\mathbf{F}_{\mathcal{V}}(X)$. We also denote by $\mathsf{V}$ the category associated to the variety.

If $l:X\rightarrow Y$ is a function, it is well known that there is a unique homomorphism $\mathbf{F}_{\mathcal{V}}(l):\mathbf{F}_{\mathcal{V}}(X)\rightarrow \mathbf{F}_{\mathcal{V}}(Y)$ that extends the map $\alpha(x)=l(x)$. If $\vec{x}\in X^{n}$ and $\varphi(\vec{x})$ is an $\mathcal{L}$-formula with variables in $\vec{x}$, then it can be proved (by induction on $\mathcal{L}$-formulas) that $\mathbf{F}_{\mathcal{V}}(l)(\varphi(\vec{x}))=\varphi(l(\vec{x}))$, where $l(\vec{x})$ denotes the $n$-tuple $(l(x_1),...,l(x_n))\in Y^{n}$. Thus the assignments $X\mapsto \mathbf{F}_{\mathcal{V}}(X)$ and $l\mapsto \mathbf{F}_{\mathcal{V}}(l)$ determine a functor $\mathbf{F}_{\mathcal{V}}$ from $\mathcal{V}$ to the category of sets and functions $\mathsf{Set}$, which is left adjoint to the forgetful functor $U:\mathsf{V}\rightarrow \mathsf{Set}$. 
  

Now, we consider the following diagram: 

\begin{displaymath}
\xymatrix{
\mathsf{WHB} \ar[rr]^-{T} & & \mathsf{TBA} 
\\  
 & \mathsf{Set}  \ar[ul]^-{\mathbf{F}_{\mathcal{WHB}}} \ar[ur]_-{\mathbf{F}_{\mathcal{TBA}}} &
}
\end{displaymath}

Let $\mathbf{A}$ be a WHB-algebra and consider its free tense extension $T(\mathbf{A})$. Notice that from Corollary \ref{subcategory of WHB}, it is the case that
\[ MT(\mathbf{A})=(\mathcal{B}({\bf{A}}), \cup, \cap, ^{c},\Rightarrow_{R_{{\bf{A}}}},\Leftarrow_{S_{{\bf{A}}}},  \emptyset, X({\bf{A}})), \]
where $U\Rightarrow_{R_{\mathbf{A}}} V = G_{{\bf{A}}}(U^{c} \cup V)$ and $U\Leftarrow_{S_{\mathbf{A}}} V = P_{{\bf{A}}}(U\backslash V)$, for every $U,V\in \mathcal{B}({\bf{A}})$, is the $\mathsf{WHB}^{\ast}$-reflection of $\mathbf{A}$. From this fact it is no hard to see that $M(\mathbf{F}_{\mathcal{WHB^{\ast}}}(X))$ is isomorphic to $\mathbf{F}_{\mathcal{TBA}}(X)$.  
\\

The next result shows that such a diagram is essentially commutative, in the sense of free tense algebras can be described by means of the functors $T$ and $\mathbf{F}_{\mathcal{WHB}}$. 
\\

From now on, we will denote by $\mathcal{L}$ the language of type $\{\wedge,\vee,\rightarrow,\leftarrow,0,1\}$, by $\mathcal{L}'$ the language $\mathcal{L}\cup \{\neg\}$ and by $\mathcal{L}_{t}$ the language of type $\{\wedge, \vee, \neg, G, H , 0, 1\}$.

\begin{theorem}\label{preservacion de libres}
The functors $T\mathbf{F}_{\mathcal{WHB}}$ and $\mathbf{F}_{\mathcal{TBA}}$ are naturally isomorphic.
\end{theorem}
\begin{proof}
We start by clarifying some notation. For all along this proof, we set $\sigma_{\mathcal{L}}$ and $\sigma_{\mathcal{L}'}$ to denote the units of $\mathbf{F}_{\mathcal{WHB}}(X)$ and $\mathbf{F}_{\mathcal{TBA}}(X)$, respectively. Let $\theta_{t}$ be the congruence of $\mathbf{Fm}_{\mathcal{L}'}(X)$ that corresponds to the equational theory of tense algebras. From general reasons, $\mathbf{Fm}_{\mathcal{L}'}(X)/\theta_{t}$ is isomorphic to $\mathbf{F}_{\mathcal{TBA}}(X)$.  Now let $h:\mathbf{F}_{\mathcal{WHB}}(X)\rightarrow \mathbf{F}_{\mathcal{WHB^{\ast}}}(X)$ be the homomorphism of WHB-algebras that extends the assignment $\alpha(x)=x/\theta_{t}$, for every $x\in X$.  Observe that $M(\mathbf{F}_{\mathcal{WHB^{\ast}}}(X))$ is isomorphic to $\mathbf{F}_{\mathcal{TBA}}(X)$, thus from Theorem \ref{la extension temporal}, there exists a homomorphism of tense algebras $h^{\ast}:T\mathbf{F}_{\mathcal{WHB}}(X)\rightarrow \mathbf{F}_{\mathcal{TBA}}(X)$. It is well known that, $h^{\ast}$ is the composition of the counit of the adjunction of Theorem \ref{la extension temporal} with $T(h)$, so from the commutativity of diagram (\ref{diagrama previo}), we get that $h^{\ast}(\sigma_{\mathcal{L}}(\varphi))=\varphi/\theta_{t}$, for every $\varphi\in \mathbf{F}_{\mathcal{WHB}}(X)$.  
\\


Consider now, the homomorphism of tense algebras $g:\mathbf{F}_{\mathcal{TBA}}(X) \rightarrow T(\mathbf{F}_{\mathcal{WHB}}(X))$ that extends the assignment $\beta(x)=\sigma_{\mathcal{L}}(x)$, for every $x\in X$. 
\begin{equation}\label{diagrama previo}
\xymatrix{
\mathbf{F}_{\mathcal{WHB}}(X) \ar[r]^-{\sigma_{\mathcal{L}}} \ar[d]_-{h} & T(\mathbf{F}_{\mathcal{WHB}}(X)) \ar[d]^-{T(h)}
\\
\mathbf{F}_{\mathcal{TBA}}(X) \ar[r]_-{\sigma_{\mathcal{L}'}} & T(\mathbf{F}_{\mathcal{TBA}}(X))
}
\end{equation}

Since $\sigma_{\mathcal{L}'}$ is bijective, it is clear that the map $f=\sigma_{\mathcal{L}'}^{-1}T(h)$ is a homomorphism of tense algebras. Further, we stress that the commutativity of diagram (\ref{diagrama previo}) implies that $fg=id_{\mathbf{F}_{\mathcal{TBA}}(X)}$ and $f(\sigma_{\mathcal{L}}(\varphi))=h(\varphi/\theta)$, for every $\mathcal{L}$-formula $\varphi$. Therefore, from the latter we conclude that $gf=id_{T(\mathbf{F}_{\mathcal{WHB}}(X))}$. Hence $T(\mathbf{F}_{\mathcal{WHB}}(X))$ and $\mathbf{F}_{\mathcal{TBA}}(X)$ are isomorphic, as desired.
\\

In order to prove the naturality, let $l:X\rightarrow Y$ be a function and consider the following diagram, in where the horizontal arrows denote the isomorphisms we previously obtained between $\mathbf{F}_{\mathcal{TBA}}(X)$ and $T\mathbf{F}_{\mathcal{WHB}}(X)$, and $\mathbf{F}_{\mathcal{TBA}}(Y)$ and $T\mathbf{F}_{\mathcal{WHB}}(Y)$, respectively:
\begin{displaymath}
\xymatrix{
\mathbf{F}_{\mathcal{TBA}}(X) \ar[r]^-{g_{X}} \ar[d]_-{\mathbf{F}_{\mathcal{TBA}}(h)} & T\mathbf{F}_{\mathcal{WHB}}(X) \ar[d]^-{T\mathbf{F}_{\mathcal{WHB}}(h)}
\\
\mathbf{F}_{\mathcal{TBA}}(Y) \ar[r]_-{g_{Y}} & T\mathbf{F}_{\mathcal{WHB}}(Y)
}
\end{displaymath}

Let $\varphi(\vec{x})$ be an $\mathcal{L}'$-formula with variables in $\vec{x}\in X^{n}$. Then, if $\sigma_X$ and $\sigma_Y$ denote the units of $\mathbf{F}_{\mathcal{TBA}}(X)$ and $\mathbf{F}_{\mathcal{TBA}}(Y)$, respectively, by the commutativity of diagram (\ref{diagrama previo}) we get
\[
\begin{array}{ccl}
g_{Y}\mathbf{F}_{\mathcal{TBA}}(l)(\varphi(\vec{x})) & = & g_{Y}(\varphi(l(\vec{x})))
\\
 & = & \sigma_{Y}(\varphi(l(\vec{x})))
\\
 & = & T\mathbf{F}_{\mathcal{WHB}}(l)(\sigma_{X}(\varphi(l(\vec{x}))))
\\
 & = &  T\mathbf{F}_{\mathcal{WHB}}(l)(g_{X}(\varphi(\vec{x}))).
\end{array}
\]
Hence $T\mathbf{F}_{\mathcal{WHB}}$ and $\mathbf{F}_{\mathcal{TBA}}$ are naturally isomorphic, as claimed.
\end{proof}

\begin{corollary}\label{Iso libres WHBast}
The functors $MT\mathbf{F}_{\mathcal{WHB}}$ and $\mathbf{F}_{\mathcal{WHB}^{\ast}}$ are naturally isomorphic.
\end{corollary}

\begin{lemma}
For every $\mathcal{L}'$-formula $\varphi$, there exists an $\mathcal{L}'$-formula $\psi$ which is a finite conjunction of formulas of the form $\delta \vee \neg \varepsilon$, where $\delta$ and $\epsilon$ are $\mathcal{L}$-formulas such that $\varphi \approx \psi$ is valid in every WHB$^{\ast}$-algebra. Moreover, the variables in $\psi$ occur all in $\varphi$.
\end{lemma}
\begin{proof}
Let $\varphi \in \mathbf{F}_{\mathcal{WHB}^{\ast}}(X)$. Observe that every element of $MT\mathbf{F}_{\mathcal{WHB}}(X)$ can be written as $\bigcap_{j\leq k}((X(\mathbf{F}_{\mathcal{WHB}}(X))\backslash \sigma(\delta_{j})) \cup \sigma(\varepsilon_{j}))$ for some $\delta_{j},\varepsilon_{j} \in \mathbf{F}_{\mathcal{WHB}}(X)$, $1\leq j\leq k$ and $k\in \mathbb{N}$. By Corollary \ref{Iso libres WHBast}, it also can be noticed that $M(g_X)=g_X$ is an isomorphism from $\mathbf{F}_{\mathcal{WHB}^{\ast}}(X)$ to $MT\mathbf{F}_{\mathcal{WHB}}(X)$. Thus, since $g_{X}(\varphi)\in MT\mathbf{F}_{\mathcal{WHB}}(X)$, $g_{X}(\psi)=\sigma_{X}(\psi)$ for every $\psi\in \mathbf{F}_{\mathcal{WHB}}(X)$ and $g_{X}$ is a morphism of tense algebras, we get that 
\[g_{X}(\varphi)=g_{X}(\bigwedge_{j\leq k}(\neg \delta_{j} \vee \varepsilon_{j})),\]
so the equation $\varphi \approx \bigwedge_{j\leq k}(\neg \delta_{j} \vee \varepsilon_{j})$ holds in $\mathbf{F}_{\mathcal{TBA}}(X)$ and consequently it holds in every tense algebra. 
\end{proof}

The proof of the following result is analogue to the one of Corollary 5.16 of \cite{CJ2}.

\begin{corollary}\label{coro simplificar formulas}
Let $\varphi \approx \psi$ be an $\mathcal{L}'$-equation. Then, there is a finite set of  $\mathcal{L}$-equations $\Pi$ of the form $\delta \wedge \epsilon \approx \delta$ and with all their variables in $\varphi \approx \psi$ such that $\mathbf{A}\models \varphi \approx \psi$ if and only if $\mathbf{A}\models \Pi$, for every WHB$^{\ast}$-algebra $\mathbf{A}$.  
\end{corollary}

\subsection{Tense companions} \label{S 6.1}

Let $\mathcal{K}$ be a variety of WHB-algebras and $\mathcal{N}$ be a variety of tense algebras. We consider the following classes: $T(\mathcal{K})=\{T(\mathbf{A})\colon \mathbf{A}\in \mathcal{K}\}$ and  $M(\mathcal{N})=\{M(\mathbf{B})\colon \mathbf{B}\in \mathcal{N}\}$.  

\begin{definition}\label{def tense companion}
We say that a variety $\mathcal{N}$ of tense algebras is a tense companion of a variety $\mathcal{K}$ of WHB-algebras if $T(\mathcal{K}) \subseteq \mathcal{N}$ and $M(\mathcal{N})\subseteq \mathcal{K}$.
\end{definition}

The proofs of the following two Lemmas are similar those of Lemmas 5.6 and 5.7 of \cite{CJ2}, so we omit them.

\begin{lemma}
Let $\mathcal{K}$ be a variety of WHB-algebras. Then $\mathcal{K}$ has a tense companion if and only if $M(T(\mathcal{K}))\subseteq \mathcal{K}$.
\end{lemma}

\begin{lemma}\label{lema I(T(K))}
If a variety $\mathcal{K}$ of WHB-algebras has a tense companion, it is $\mathbb{I}(T(\mathcal{K}))$. Moreover, $\mathcal{K}$ has a modal companion if and only if $\mathbb{I}(T(\mathcal{K}))$ is a variety.
\end{lemma}

\begin{proposition}\label{tense companions}
The variety of WHB-algebras has a tense companion. 
\end{proposition}
\begin{proof}
It follows from Lemma \ref{B(A) es tense algebra}, the fact that  $\mathbb{I}(T(\mathcal{WHB}))=\mathcal{TBA}$ and Lemma \ref{lema I(T(K))}.
\end{proof}

We say that a WHB-algebra $\mathbf{A}$ is a \emph{Basic WHB-algebra} if the equations (B) and (B$^{\ast}$) hold in $\mathbf{A}$. It follows from Corollary \ref{subvariedades de WHB} it follows that $\bf{A}$ is a basic WHB-algebra if some of the conditions $\rm{B}$ or $\mathrm{B}^{*}$ is valid on $\bf{A}$. It is clear that the class of Basic WHB-algebras form a variety which will be denoted by $\mathcal{BWHB}$ and it is clear that $\mathcal{BWHB} \subseteq \mathcal{WHB}$. 

\begin{proposition}
Neither the variety $\mathcal{BWHB}$ nor the variety $\mathcal{WHB}$ have a tense companion. Moreover, the classes $\mathbb{I}(T(\mathcal{BWHB}))$ and $\mathbb{I}(T(\mathcal{WHB}))$ are not varieties.
\end{proposition}
\begin{proof}
Let us consider the Basic WH-algebra $\mathbf{A}$ defined by the three element chain $0<a<1$. Notice that $X(\mathbf{A})=\{P_1,P_2\}$, where $P_1=\{1\}$ and $P_2=\{1,a\}$. Further is true, $R_{A}=\{(P_1,P_1),(P_1,P_2),(P_2,P_2)\}$ and $S_{A}=R_{A}^{-1}$ by Lemma \ref{ecuaciones y relaciones}. Hence, $\{P_1\}\nsubseteq X(\mathbf{A})\Rightarrow_{R_{A}}\{P_1\}$ and $\{P_1\}\Leftarrow_{S_{A}}\emptyset\nsubseteq \{P_1\}$. Therefore, we have shown that although the equations (B) and (B$^{\ast}$) hold in $\mathbf{A}$, they do not hold in its tense extension. The moreover part follows from Lemma \ref{lema I(T(K))}.
\end{proof}

Recall that a variety has the \emph{finite model property} if every equation which is not valid in some member of the variety is not valid in a finite member, or, equivalently, if it is generated by its finite members. 

\begin{theorem}\label{finite model property}
Let $\mathcal{K}$ be a variety of WHB-algebras with a tense companion. Then, $\mathcal{K}$ has the finite model property iff its tense companion has the finite model property.
\end{theorem}
\begin{proof}
Let $\mathcal{N}$ be the tense companion of $\mathcal{K}$. One the one hand, if $\mathcal{N}$ has the finite model property, then there exists a finite $\mathbf{B}\in \mathcal{N}$ and an $\mathcal{L}_{t}$-equation $\varphi\approx \psi$ such that $\mathbf{B}\nvDash \varphi\approx \psi$. Thus, there exists an $\mathcal{L}'$-equation $\varphi'\approx \psi'$ such that $M(\mathbf{B})\nvDash \varphi'\approx \psi'$. It is worth noticing that such an equation is just the $\mathcal{L}'$-equation that comes up after applying the interpretations of $G$, $H$, in terms of $\leftarrow$, $\rightarrow$ and $\neg$, to the equation $\varphi\approx \psi$. From Corollary \ref{coro simplificar formulas}, there exists a finite set $\Pi$ of $\mathcal{L}$-formulas such that $M(\mathbf{B})\nvDash \Pi$. So, there exists an $\mathcal{L}$-equation $\epsilon\approx \delta \in \Pi$ such that $M(\mathbf{B})\nvDash \epsilon\approx \delta$. Since $M(\mathbf{B})\in \mathcal{K}$ by assumption, and it is finite, then we conclude that $\mathcal{K}$ has the finite model property. On the other hand, if $\mathcal{K}$ has the finite model property, there exists a finite $\mathbf{A}\in \mathcal{K}$ and an $\mathcal{L}$-equation $\epsilon\approx \delta$ such that $\mathbf{A}\nvDash \epsilon\approx \delta$. Observe that since $\sigma_{\mathbf{A}}$ is an embedding, it is the case that $MT(\mathbf{A})\nvDash \epsilon\approx \delta$. Let $\varphi\approx \psi$ be the $\mathcal{L}'$-equation that arises after applying the interpretations of $\leftarrow$ and $\rightarrow$, in terms of $G$, $H$ and $\neg$. Then it follows that $T(\mathbf{A})\nvDash \varphi\approx \psi$. Since $T(\mathbf{A})$ is a finite member of $\mathcal{N}$, from the latter, we get that $\mathcal{N}$ has the finite model property. This concludes the proof.
\end{proof}

It is well known, that tense algebras has the finite model property (\cite{K1998}). Thus straightforward from Theorem \ref{finite model property}, the following holds.

\begin{corollary}
The variety $\mathcal{WHB}$ has the finite model property. 
\end{corollary}

A variety $\mathcal{V}$ has \emph{definable principal congruences} (DPC) if there exists a formula $\zeta(x, y, u, v)$ in the first-order language of $\mathcal{V}$ such that for every $\mathbf{A}\in \mathcal{V}$ and all $a, b, c, d, \in A$ we have:
\begin{displaymath}
\begin{array}{ccc}
(c,d)\in \mathsf{Cg}^{\mathbf{A}}(a,b) & \Longleftrightarrow & \mathbf{A} \models \zeta[a, b, c, d].
\end{array}
\end{displaymath}
If $\zeta(x, y, u, v)$ happens to be a finite conjunction of equations, then $\mathcal{V}$ is said to have \emph{equationally definable principal congruences} (EDPC).

\begin{proposition}\label{EDPC tense companion}
Let $\mathcal{K}$ be a variety of WHB-algebras with a tense companion. Then, $\mathcal{K}$ has EDPC iff its tense companion has EDPC.
\end{proposition}
\begin{proof}
Let $\mathcal{N}$ be the tense companion of $\mathcal{K}$. On the one hand, let us assume that $\mathcal{K}$ has EDPC with $\zeta(x,y,u,v) = \bigwedge_{i\leq n} \varphi_{i}(x,y,u,v)\approx \psi_{i}(x,y,u,v)$, and let $\mathbf{B}\in \mathcal{N}$ and $p,q\in B$. If $(s,t)\in \mathsf{Cg}^{\mathbf{B}}(p,q)$, by Corollary \ref{lema EDPC}, we get that $(s,t)\in \mathsf{Cg}^{M(\mathbf{B})}(p,q)$. Now, since $M(\mathcal{N})\subseteq \mathcal{K}$ by assumption, the latter is equivalent to say that $\varphi_{i}^{M(\mathbf{B})}(s,t,p,q)=\psi_{i}^{M(\mathbf{B})}(s,t,p,q)$, for every $1\leq i\leq n$. Then, if we take $\lambda_{i}(x,y,u,v)$ and $\mu_{i}(x,y,u,v)$ to be the $\mathcal{L}_{t}$-terms resulting from applying the interpretations of $\leftarrow$ and $\rightarrow$ in terms of $G$, $H$ and $\neg$, in all its occurrences in $\varphi_{i}$ and $\psi_{i}$, respectively, then we are able to conclude that $(s,t)\in \mathsf{Cg}^{\mathbf{B}}(p,q)$ if and only if $\lambda_{i}^{\mathbf{B}}(s,t,p,q)=\mu_{i}^{\mathbf{B}}(s,t,p,q)$, for every $1\leq i\leq n$. Hence, $\mathcal{N}$ has EDPC, as claimed. On the other hand, Let us assume that $\mathcal{N}$ has EDPC with $\eta(x,y,u,v) = \bigwedge_{i\leq m} \lambda_{i}(x,y,u,v)\approx \mu_{i}(x,y,u,v)$. Let $\mathbf{A}\in \mathcal{K}$ and suppose $(c,d)\in \mathsf{Cg}^{\mathbf{A}}(a,b)$. Then from Corollary \ref{lema EDPC}, this is equivalent to $(\sigma_{\mathbf{A}}(c),\sigma_{\mathbf{A}}(d))\in \mathsf{Cg}^{T(\mathbf{A})}(\sigma_{\mathbf{A}}(a),\sigma_{\mathbf{A}}(b))$. So, since $T(\mathcal{K})\subseteq \mathcal{N}$ by hypothesis, the latter is also equivalent to 
\[\lambda^{T(\mathbf{A})}_{i}(\sigma_{\mathbf{A}}(c),\sigma_{\mathbf{A}}(d),\sigma_{\mathbf{A}}(a),\sigma_{\mathbf{A}}(b))= \mu^{T(\mathbf{A})}_{i}(\sigma_{\mathbf{A}}(c),\sigma_{\mathbf{A}}(d),\sigma_{\mathbf{A}}(a),\sigma_{\mathbf{A}}(b)),\]
for every $1\leq i\leq m$. If we set $\varphi_{i}(x,y,u,v)$ and $\psi_{i}(x,y,u,v)$ to be the $\mathcal{L}'$-terms resulting from applying the interpretations of $G$, $H$ in terms of $\leftarrow$, $\rightarrow$ and $\neg$, in all its occurrences in $\lambda_{i}$ and $\mu_{i}$, respectively, then we get
\begin{equation}\label{aux equation}
\varphi^{MT(\mathbf{A})}_{i}(\sigma_{\mathbf{A}}(c),\sigma_{\mathbf{A}}(d),\sigma_{\mathbf{A}}(a),\sigma_{\mathbf{A}}(b))= \psi^{MT(\mathbf{A})}_{i}(\sigma_{\mathbf{A}}(c),\sigma_{\mathbf{A}}(d),\sigma_{\mathbf{A}}(a),\sigma_{\mathbf{A}}(b)),
\end{equation}

for every $1\leq i\leq n$. Observe that due to Corollary \ref{coro simplificar formulas}, for every $\mathcal{L}'$-equation $\varphi_{i}(x,y,u,v)\approx \psi_{i}(x,y,u,v)$, there exist a finite set of $\mathcal{L}$-equations $\Pi_{i}$ (with variables in $x,y,u,v$), such that (\ref{aux equation}) holds if and only if 
\[ 
MT(\mathbf{A}) \models \Pi_{i}[\sigma_{\mathbf{A}}(c),\sigma_{\mathbf{A}}(d),\sigma_{\mathbf{A}}(a),\sigma_{\mathbf{A}}(b)],
\]
which in turn is equivalent to $\mathbf{A} \models \Pi_{i}[c,d,a,b],$ because $\sigma_{\mathbf{A}}$ is an embedding. Let us take $\zeta(x,y,u,v) = \bigwedge_{i\leq m} \Pi_{i}$. It is clear then that $(c,d)\in \mathsf{Cg}^{\mathbf{a}}(a,b)$ if and only if $\mathbf{A} \models \zeta[c,d,a,b]$. I.e., $\mathcal{K}$ has EDPC, as granted.
\end{proof}

The \emph{ternary discriminator} is a function $t(x, y, z)$, such that:
\begin{displaymath}
t(x,y,z)=\begin{cases} 
      x, & \text{if}\; x\neq y \\
      z, & \text{if}\; x= y. 
   \end{cases}
\end{displaymath}
A \emph{discriminator variety} is a variety $\mathcal{V}$ such that the ternary discriminator is a term-operation on every subdirectly irreducible algebra in $\mathcal{V}$. A variety $\mathcal{V}$ is \emph{semisimple} if and only if all subdirectly irreducible members of $\mathcal{V}$ are simple.

\begin{proposition}\label{equivalences EDPC}
Let $\mathcal{K}$ be a variety of WHB-algebras with a tense companion $\mathcal{N}$. Then, the following are equivalent:
\begin{enumerate}[\normalfont (1)]
\item $\mathcal{K}$ has EDPC if and only if $\mathcal{N}$ has EDPC.
\item $\mathcal{K}$ is a discriminator variety if and only if $\mathcal{N}$ is a discriminator variety.
\item $\mathcal{K}$ is semisimple if and only if $\mathcal{N}$ is semisimple.
\end{enumerate} 
\end{proposition}
\begin{proof}
Use Proposition \ref{EDPC tense companion} and Theorem 2.4 of \cite{K1998}.
\end{proof}

Let us consider the unary $\mathcal{L}_{t}$-term $d(x)=x\wedge G(x)\wedge H(x)$. We set $d^{0}(x)=x$ and $d^{n+1}(x)=d(d^{n}(x))$. We say that $d$ is cyclic if there is a $n\in \mathbb{N}$ such that $d^{n+1}(x)=d^{n}(x)$. Let $n\in \mathbb{N}$ and consider the $\mathcal{L}_{t}$-equation $d^{n+1}(x)\approx d^{n}(x)$. It is clear from the interpretations of $G$ and $H$ by means of $\rightarrow$, $\leftarrow$ and $\neg$, that the latter induces an $\mathcal{L}'$-equation $\varphi^{n}(x)\approx \varphi^{n+1}(x)$. So, from Corollary \ref{coro simplificar formulas}, there exists a finite set $\Pi_{n}(x)$ of $\mathcal{L}$-equations such that $\mathbf{C}\models \varphi^{n}(x)\approx \varphi^{n+1}(x)$ if and only if $\mathbf{C}\models \Pi_{n}(x)$, for every WHB$^{\ast}$-algebra $\mathbf{C}$. Therefore, as a consequence of Proposition \ref{equivalences EDPC} (1) and Theorem 2.4 (iv) of \cite{K1998}, we have proved: 

\begin{corollary}\label{equiv2 EDPC}
Let $\mathcal{K}$ be a variety of WHB-algebras with a tense companion. Then, the following is equivalent: 
\begin{enumerate}[\normalfont (1)]
\item $\mathcal{K}$ has EDPC.
\item There is $n\in \mathbb{N}$ such that $\Pi_{n}(x)$ holds in the $\mathcal{K}$.
\item There is $n\in \mathbb{N}$ such that $T(\mathbf{A})\models d^{n+1}(x)\approx d^{n}(x)$, for all $\mathbf{A}\in \mathcal{K}$.
\end{enumerate}
\end{corollary}

The following result follows from the fact that tense extensions of Heyting Brouwer algebras are S$_{4}$-tense algebras \cite{Wolter}, and in these algebras, the $\mathcal{L}_{t}$-term $d(x)$ is not cyclic in general. So, from Corollary \ref{equiv2 EDPC} we may conclude: 

\begin{corollary}
The variety of double Heyting algebras is not semisimple, nor a discriminator variety nor has EDPC.
\end{corollary}

A variety $\mathcal{V}$ is said to have the \emph{amalgamation property} (AP, for short) if for every span (or pair of homomorphisms) $\mathbf{B} \xleftarrow{i} \mathbf{A} \xrightarrow{j} \mathbf{C}$, with $i$ and $j$ monomorphisms, there exists $\mathbf{D} \in \mathcal{V}$ and monomorphisms $\mathbf{B} \xrightarrow{h} \mathbf{D} \xleftarrow{k} \mathbf{C}$ such that $hi=kj$.

\begin{theorem}\label{Amalgamation property}
Let $\mathcal{K}$ be a variety of WHB-algebras with a tense companion. Then, $\mathcal{K}$ has AP if and only if its tense companion has AP.
\end{theorem}
\begin{proof}
Let $\mathcal{N}$ be the tense companion of $\mathcal{K}$. On the one hand, let us assume $\mathcal{N}$ has AP and let $\mathbf{B} \xleftarrow{i} \mathbf{A} \xrightarrow{j} \mathbf{C}$ be a span in $\mathcal{K}$, with $i$ and $j$ monomorphisms. Then, from applying $T$ to the span and the assumption on $\mathcal{N}$, there exists $\mathbf{D} \in \mathcal{N}$ and monomorphisms $T(\mathbf{B}) \xrightarrow{h} \mathbf{D} \xleftarrow{k} M(\mathbf{C})$, such that $kT(j)=hT(i)$. Notice that by Theorem \ref{la extension temporal}, we get that $T(i)$ and $T(j)$ are monomorphisms. Now, by applying $M$ and regarding the unit $\sigma$ of the adjunction $T\dashv M$ (Theorem \ref{la extension temporal}), we obtain the following diagram in $\mathsf{WHB}$:
\begin{displaymath}
\xymatrix{
\mathbf{A} \ar[rr]^-{j} \ar[dd]_-{i} \ar[dr]^-{\sigma_{\mathbf{A}}} & & \mathbf{C} \ar[dr]^-{\sigma_{\mathbf{C}}} &
\\
 & MT(\mathbf{A}) \ar[rr]^-{MT(j)} \ar[dd]_-{MT(i)} & & MT(\mathbf{C}) \ar[dd]^-{M(k)}
\\
\mathbf{B} \ar[dr]_-{\sigma_{\mathbf{B}}}  & & & 
\\
 & MT(\mathbf{B}) \ar[rr]_-{M(h)} & & M(\mathbf{D})
}
\end{displaymath}
Notice that from the functoriality of $MT$ and the naturality of $\sigma$, all the squares in the above diagram commute. Moreover, since $M$ is a right adjoint, then $M(k)$ and $M(h)$ are mono. Thus, is we take $u=M(h)\sigma_{\mathbf{B}}$ and $u=M(k)\sigma_{\mathbf{C}}$, it is clear that $ui=vj$ and that $u$ and $v$ are mono. Hence, $\mathcal{K}$ has AP, as desired.

On the other hand, suppose $\mathcal{K}$ has AP and let $\mathbf{G} \xleftarrow{t} \mathbf{E} \xrightarrow{s} \mathbf{F}$ be a span in $\mathcal{N}$, with $s$ and $t$ monomorphisms. Thus, from the application of $M$ on the span and the hypothesis on $\mathcal{N}$, there exists $\mathbf{H} \in \mathcal{N}$ and monomorphisms $M(\mathbf{G}) \xrightarrow{\alpha} \mathbf{H} \xleftarrow{\beta} M(\mathbf{F})$, such that $\beta M(s)=\alpha M(t)$. Since $M$ is a left adjoint, $M(s)$ and $M(t)$ are monomorphisms. Now, from the application of $T$ and considering the counit $\varepsilon$ of the adjunction  $T\dashv M$ (Theorem \ref{la extension temporal}), we get the following diagram in $\mathsf{TBA}$:
\begin{displaymath}
\xymatrix{
\mathbf{E} \ar[rr]^-{s} \ar[dd]_-{t} \ar[dr]^-{\varepsilon_{\mathbf{E}}} & & \mathbf{F} \ar[dr]^-{\varepsilon_{\mathbf{F}}} &
\\
 & TM(\mathbf{A}) \ar[rr]^-{TM(s)} \ar[dd]_-{TM(t)} & & TM(\mathbf{F}) \ar[dd]^-{T(\beta)}
\\
\mathbf{B} \ar[dr]_-{\varepsilon_{\mathbf{G}}}  & & & 
\\
 & TM(\mathbf{G}) \ar[rr]_-{T(\alpha)} & & T(\mathbf{H})
}
\end{displaymath}
We stress that all the squares in the above diagram commute because of the functoriality of $TM$ and the naturality of $\varepsilon$. Further, from Theorem \ref{la extension temporal}, it also follows that $TM(t)$ and $TM(s)$ are monomorphisms. So, if we set $\lambda=T(\alpha)\varepsilon_{\mathbf{G}}$ and $\mu=T(\beta)\varepsilon_{\mathbf{F}}$, it is straightforward to conclude that $\mathcal{N}$ has AP, as claimed.
\end{proof}

In \cite{LAG} it was proved that tense logic has the interpolation property. Taking into account that tense algebras are the algebraic counterpart for tense logics it follows that tense algebras has AP. Therefore, as consequence of Theorem \ref{Amalgamation property}, we get:

\begin{corollary}\label{WHB has AP}
The variety $\mathcal{WHB}$ has AP.
\end{corollary}

A class of algebras $\mathcal{V}$ has the \emph{congruence extension property} (CEP, for short) if for every subalgebra $\mathbf{A}$ of $\mathbf{B} \in \mathcal{V}$ and each congruence $\delta$ of $\mathbf{A}$, there exists a congruence $\theta$ of $\mathbf{B}$ such that $\delta=\theta \cap A^{2}$. It is well known that CEP is equivalent to the following condition: for every span $\mathbf{B} \xleftarrow{i} \mathbf{A} \xrightarrow{j} \mathbf{C}$ in $\mathcal{V}$ such that $i$ is injective and $j$ is surjective, there exists $\mathbf{D} \in \mathcal{V}$ and homomorphisms $\mathbf{B} \xrightarrow{h} \mathbf{D} \xleftarrow{k} \mathbf{C}$, such that $h$ is surjective, $k$ is injective and $kj=hi$.

\begin{theorem}\label{tense CEP}
Let $\mathcal{K}$ be a variety of WHB-algebras with a tense companion. Then, $\mathcal{K}$ has CEP if and only if its tense companion has CEP.
\end{theorem}
\begin{proof}
Let $\mathcal{N}$ be the tense extension of $\mathcal{K}$. On the one hand, suppose $\mathcal{K}$ has CEP, and let $\mathbf{G} \xleftarrow{t} \mathbf{E} \xrightarrow{s} \mathbf{F}$ be a span in $\mathcal{N}$, with $t$ injective and $s$ surjective. Then from applying $M$ to the span and the assumption on $\mathcal{K}$, there exists $\mathbf{H} \in \mathcal{N}$ and homomorphisms $M(\mathbf{G}) \xrightarrow{\alpha} \mathbf{H} \xleftarrow{\beta} M(\mathbf{F})$, with $\beta$ surjective and $\alpha$ injective such that $\beta M(s)=\alpha M(t)$. Notice that, from the application of $T$ and considering the counit $\varepsilon$ of the adjunction  $T\dashv M$ (Theorem \ref{la extension temporal}), we get a diagram in $\mathsf{TBA}$ analogue to the second diagram of the proof of Theorem \ref{Amalgamation property}. Observe that $T(\alpha)$ is injective and $T(\beta)$ is surjective due to Theorem \ref{la extension temporal}. Therefore, since $\varepsilon_{\mathbf{F}}$ and $\varepsilon_{\mathbf{G}}$ are isomorphisms, we may conclude that $\lambda=T(\alpha)\varepsilon_{\mathbf{G}}$ is injective, $\mu=T(\beta)\varepsilon_{\mathbf{F}}$ is surjective and $\mu s=\lambda t$. So, $\mathcal{N}$ has CEP, as desired. 

On the other hand, suppose $\mathcal{N}$ has CEP and let $\mathbf{A}, \mathbf{B}\in \mathcal{K}$ be such that $\mathbf{A}$ is subalgebra of $\mathbf{B}$. If we write $i$ for the inclusion of $A$ in $B$, then from Theorem \ref{la extension temporal}, $T(i):T(\mathbf{A})\rightarrow T(\mathbf{B})$ is an embedding, so $T(\mathbf{A})$ is isomorphic to a subalgebra $\mathbf{C}$ of $T(\mathbf{B})$. We stress that  $\mathbf{C}\in \mathcal{N}$, since $T(\mathbf{B})\in \mathcal{N}$. Further, notice that it is also true that the map defined as $g(\sigma_{\mathbf{A}}(a),\sigma_{\mathbf{A}}(b))=(T(i)(\sigma_{\mathbf{A}}(a)),T(i)(\sigma_{\mathbf{A}}(b)))$, induces an isomorphism between $\mathsf{Con}_{\mathcal{TBA}}(T(\mathbf{A}))$ and $\mathsf{Con}_{\mathcal{TBA}}(\mathbf{C})$. Let us write $\Phi_{\mathbf{A}}$ and $\Phi_{\mathbf{B}}$ be the isomorphisms of Theorem \ref{Con(A) iso Con(T(A))} between $\mathsf{Con}_{\mathcal{WHB}}(\mathbf{A})$ and $\mathsf{Con}_{\mathcal{TBA}}(T(\mathbf{A}))$, and $\mathsf{Con}_{\mathcal{WHB}}(\mathbf{B})$ and $\mathsf{Con}_{\mathcal{TBA}}(T(\mathbf{B}))$, respectively. 
If we take $\delta \in \mathsf{Con}_{\mathcal{WHB}}(\mathbf{A})$, then $g(\Phi_{\mathbf{A}}(\delta))\in \mathsf{Con}_{\mathcal{TBA}}(\mathbf{C})$, so from hypothesis on $\mathcal{N}$ and Theorem \ref{la extension temporal} there is a $\theta \in \mathsf{Con}_{\mathcal{WHB}}(\mathbf{B})$, such that $g(\Phi_{\mathbf{A}}(\delta))=\Phi_{\mathbf{B}}(\theta)\cap C^{2}$. We will proof that $\delta=\theta\cap A^{2}$. To do so, notice first that from the naturality of the unit $\sigma$ of the adjunction $T\dashv M$ (Theorem \ref{la extension temporal}), it is the case that for every $a\in A$:
\begin{equation}\label{equation other}
\sigma_{\mathbf{B}}(a)=\sigma_{\mathbf{B}}(i(a))=T(i)(\sigma_{\mathbf{A}}(a)).
\end{equation}
Hence by Theorem \ref{Con(A) iso Con(T(A))} and (\ref{equation other}) it follows:
\begin{displaymath}
\begin{array}{ccl}
(a,b)\in \delta & \Leftrightarrow & (\sigma_{\mathbf{A}}(a),\sigma_{\mathbf{A}}(b)) \in \Phi_{\mathbf{A}}(\delta), 
\\
 & \Leftrightarrow & (\sigma_{\mathbf{B}}(a),\sigma_{\mathbf{B}}(b)) \in g(\Phi_{\mathbf{A}}(\delta)),
 \\
 & \Leftrightarrow & (\sigma_{\mathbf{B}}(a),\sigma_{\mathbf{B}}(b)) \in \Phi_{\mathbf{B}}(\theta) \; \text{and}\; (\sigma_{\mathbf{B}}(a),\sigma_{\mathbf{B}}(b))\in C^{2}, 
 \\
  & \Leftrightarrow & (a,b) \in \theta \; \text{and}\; (\sigma_{\mathbf{A}}(a),\sigma_{\mathbf{A}}(b))\in T(\mathbf{A})^{2}, 
 \\
 & \Leftrightarrow & (a,b) \in \theta  \cap A^{2}. 
\end{array}
\end{displaymath}
Therefore, $\mathcal{K}$ has CEP, as granted.
\end{proof}

It is well known (Fact 1.11 of \cite{K1998}) that the variety of tense algebras has CEP. Thus, from Theorem \ref{tense CEP}, the following holds.

\begin{corollary}\label{WHB has CEP}
The variety $\mathcal{WHB}$,  so any of its subvarieties, has CEP. 
\end{corollary}

We conclude this paper with by recalling that a variety $\mathcal{V}$ has the \emph{Maehara interpolation property} (MIP, for short) if for any set $Y$, whenever
\begin{itemize}
\item[(i)] $\Sigma \cup \Gamma \cup \{\varepsilon\} \subseteq Eq(Y)$ and $Var(\Sigma)\cap Var(\Gamma \cup \{\varepsilon\})\neq \emptyset$;
\item[(ii)] $\Sigma \cup \Gamma\models_{\mathcal{V}} \varepsilon$,
\end{itemize}
there exists $\Delta \subseteq Eq(Y)$ such that:
\begin{itemize}
\item[(iii)] $\Sigma \models_{\mathcal{V}} \Delta$,
\item[(vi)] $\Delta \cup \Gamma \models_{\mathcal{V}} \varepsilon$,
\item[(v)] $Var(\Delta)\subseteq Var(\Sigma) \cap Var(\Gamma \cup \{\varepsilon\})$.
\end{itemize}

In Theorem 29 of \cite{MMT2014} it was proved that a varitey $\mathcal{V}$ has MIP if and only if $\mathcal{V}$ has AP and CEP. Therefore, as consequence of Corollaries \ref{WHB has AP} and \ref{WHB has CEP} we get

\begin{corollary}\label{WHB has CEP}
The variety $\mathcal{WHB}$ has MIP. 
\end{corollary}

-------------------------------------------------------------------------------------------------------
\\
Sergio Arturo Celani, \\
Departamento de Matem\'atica,\\
Facultad de Ciencias Exactas (UNCPBA),\\
Pinto 399, Tandil (7000),\\
and CONICET, Argentina,\\
scelani@exa.unicen.edu.ar

-----------------------------------------------------------------------------------
\\
Agust\'in Leonel Nagy \\
Departamento de Matem\'atica,\\
Facultad de Ciencias Exactas (UNCPBA),
and CONICET.\\
Pinto 399, Tandil (7000),\\
agustin.nagy@gmail.com

-----------------------------------------------------------------------------------------
\\
William Javier Zuluaga Botero,\\
Facultad de Ciencias Exactas (UNCPBA),\\
Pinto 399, Tandil (7000),\\
and CONICET, Argentina,\\
wizubo@exa.unicen.edu.ar

\end{document}